\documentclass[a4paper,11pt]{amsart}
\usepackage{amsmath}
\usepackage{cases}
\usepackage{amsfonts}
\usepackage[colorlinks,linkcolor=blue,citecolor=blue]{hyperref}
\usepackage{latexsym, amssymb, amsmath, amsthm, bbm}
\usepackage[all]{xy}
\usepackage{pgfplots}
\usepackage{mathrsfs}
\usepackage [latin1]{inputenc}
\usepackage[misc]{ifsym}

\DeclareSymbolFont{EulerExtension}{U}{euex}{m}{n}
\DeclareMathSymbol{\euintop}{\mathop} {EulerExtension}{"52}
\DeclareMathSymbol{\euointop}{\mathop} {EulerExtension}{"48}

\allowdisplaybreaks[4]

\def \Ob{\operatorname{Ob}}
\def \id{\operatorname{Id}}
\def \ker{\operatorname{Ker}}
\def \ord{\operatorname{ord}}

\def \dim{\operatorname{dim}}

\def \Id{\operatorname{Id}}
\def \ord{\operatorname{ord}}
\def \Rep{\operatorname{Rep}}

\def \id{\operatorname{Id}}
\def \ker{\operatorname{Ker}}

\numberwithin{equation}{section}

\newtheorem{theorem}{Theorem}[section]
\newtheorem{lemma}[theorem]{Lemma}
\newtheorem{proposition}[theorem]{Proposition}
\newtheorem{corollary}[theorem]{Corollary}
\newtheorem{definition}[theorem]{Definition}
\newtheorem{example}[theorem]{Example}
\newtheorem{remark}[theorem]{Remark}

\newtheorem{convention}[theorem]{Convention}

\begin{document}
\title{Constructing Modular Tensor Categories with Hopf Algebras}
\author{Kun Zhou}\address{Beijing Institute of Mathematical Sciences and Applications, Beijing 101408, China. }\email{kzhou@bimsa.cn}

\subjclass[2020]{16T05, 18M20, 16T25}
\keywords{Factorizable Hopf algebra, Modular tensor category, Abelian extension, Drinfel'd double, $n$-rank Taft algebra.}

\date{}
\maketitle
\begin{abstract}
A modular tensor category $\mathcal{C}$ is a non-degenerate ribbon finite tensor category (\cite{Ker}). And a ribbon factorizable Hopf algebra is exactly the Hopf algebra whose finite-dimensional representations form a modular tensor category. The goal of this paper is to construct both semisimple and non-semisimple modular categories with Hopf algebras. In particular, we study central extensions of factorizable Hopf algebras and characterize some conditions for a quotient Hopf algebra to be a factorizable one. Then we give a special way to obtain ribbon elements. As a result, we apply to finite dimensional pointed Hopf algebras $u(\mathcal{D},\lambda,\mu)$, $\mathcal{U}(\mathcal{D}_{red})$ introduced by Andruskiewitsch etc and obtain two families of ribbon factorizable Hopf algebras $H(\omega,n),K(\alpha,n)$. With little restrictions to their parameters, we prove that $\Rep(H(\omega,n)),\Rep(K(\alpha,n))$ are prime modular tensor categories which are not tensor equivalent to representation categories of small quantum groups or the ones in \cite[Example 5.18]{LC}. Lastly, we construct a family of semisimple ribbon factorizable Hopf algebas $A(p,q)$, where $p,q$ are prime numbers satisfying some conditions. And we decompose $\Rep(A(p,q))$ into prime modular tensor categories and prove that $A(p,q)$ can't be obtained by some obvious ways, i.e. it's not tensor product of trivial Hopf algebras (group algebras or their dual) and Drinfeld doubles.
\end{abstract}

\section{Introduction}
A modular tensor category (MTC) is a braided finite tensor category with some additional algebraic structures (duality, twist, and a non-degeneracy axiom, see \cite{Ki, Ker}). It provides a topological quantum field theory in dimension 3, and in particular, invariants of links and 3-manifolds (see \cite{NV, Ker}). Moreover, unitary MTCs has been used in quantum computing (see \cite{Fre, Fre1}). How can MTCs be obtained? If we consider Hopf algebras whose finite-dimensional representations form MTCs, then we will obtain ribbon factorizable Hopf algebras (see for example \cite[Section 2.5]{Shi1}). As a result, we can consider constructing these Hopf algebras to produce MTCs. For example, small quantum groups with specific parameters and quantum doubles of finite groups are ribbon and factorizable Hopf algebras, which can thus producing MTCs. Although these examples support a fluent source of MTCs, some additional questions (see \cite[Section 1.2]{Ro}) motivate us to establish more MTCs. Therefore, we focus on constructing MTCs from Hopf algebras in this paper.

Many authors have in fact considered constructing MTCs by using methods other than Hopf algebras (\cite{Bla,Mug,HY,Ro,HS,Shi,LC}). For example, M. M\"{u}ger established a double commutant theorem about MTCs in \cite{Mug}. This implies that there is a way to construct MTCs by considering the relative M\"{u}ger center of MTCs (see \cite[Corollary 3.5]{Mug}). This method is applicable to the construction of semisimple MTCs. However, what about the construction of non-semisimple MTCs? In another paper \cite{Shi}, K. Shimizu found that the characterizations of non-degeneracy of a braided fusion category can also apply to the non-semisimple case. Using the results obtained, he established a method to construct factorizable braided tensor category by considering Yetter-Drinfeld categories (see \cite[Section 6]{Shi}), and hence providing a new source of factorizable Hopf algebras which are not Drinfel'd doubles in general. In \cite{LC}, R. Laugwitz and C. Walton develop this approach to construct MTCs. Their main idea is to use Shimizu's result (\cite[Theorem 4.9]{Shi}) about double commutant theorem, and obtain a result (\cite[Theorem 4.3]{LC}) about relative M\"{u}ger center as their main way to construct MTCs. Using this theorem and some other techniques, they not only recover the MTCs provided by the representations of small quantum groups (\cite[Example 5.17]{LC}) but also concretely give some new non-semisimple MTCs (such as \cite[Example 5.18]{LC}). And these non-semisimple MTCs are also representations of ribbon factorizable Hopf algebras. At present, representation categories of ribbon factorizable Hopf algebras are still important source of MTCs. Inspired by this fact and the above works (such as \cite[Theorem 4.3]{LC}), we try to continue exploring the construction of MTCs but using Hopf algebras. That is, we want to provide a computable way to construct ribbon and factorizable Hopf algebras, and hence producing MTCs.

Our idea is as follows. We choose a Hopf algebra $H$ that makes $D(H)$ the ribbon Hopf algebra. Then, we consider Hopf ideals of $D(H)$ satisfying the corresponding  quotient Hopf algebras are factorizable Hopf algebras. The difficult part is choosing the required ideals. To address this, we use central extension of Hopf algebras and the theorem about MTCs given in \cite[Theorem 4.3]{LC}. As a result, we establish a general way to construct factorizable Hopf algebras. Then, we give a special way handle ribbon elements by considering central group-like elements and thus provide ribbon factorizable Hopf algebras. 

To construct ribbon factorizable Hopf algebras concretely, we apply our results to the Hopf algebras $u(\mathcal{D},\lambda,\mu)$, $\mathcal{U}(\mathcal{D}_{red})$ introduced by Andruskiewitsch etc. In fact, the authors in \cite{LC} have studied $u(\mathcal{D},0,0)$, $\mathcal{U}(\mathcal{D}_{red})$ . And they have characterized conditions for braided Drinfeld double to be a ribbon factorizable Hopf algebras by using relative M\"{u}ger center. In this paper, we consider not only $u(\mathcal{D},0,0)$ but also their lifting cases. Moreover, our method to obtain ribbon elements are different from \cite{LC}. In \cite{LC}, the author mainly use the result of \cite{BS} to get ribbon elements, while we mainly use the property that there central group-like elements are trivial to obtain ribbon elements. Then we construct two families of ribbon factorizable Hopf algebras $H(\omega,n),K(\alpha,n)$ and proved that $\Rep(H(\omega,n)),\Rep(K(\alpha,n))$ are prime modular tensor categories which are not tensor equivalent to representation categories of small quantum groups or the ones in \cite[Example 5.18]{LC}. The main tool to achieve this is the method of \cite{TM,ME} to discuss quotient Hopf algebras. Since our goal, we also construct a family of semisimple ribbon factorizable Hopf algebras $A(p,q)$. Then we decompose $\Rep(A(p,q))$ into prime modular tensor categories. Although $\Rep(A(p,q))$ is not prime, we show that $A(p,q)$ can't be obtained by some obvious way, that is to say it's not tensor product of trivial Hopf algebras (group algebras or their dual) and Drinfeld doubles.

The rest of this paper is organized as follows. In Section 2, we provide some preliminaries including the definition of modular tensor categories and ribbon factorizable Hopf algebras, the extension of Hopf algebras, the pointed Hopf algebras $u(\mathcal{D},\lambda,\mu)$, $\mathcal{U}(\mathcal{D}_{red})$. Then, a family of semisimple Hopf algebras $\mathscr{A}_l$ are also recalled, and a theorem for $H$ to be a ribbon Hopf algebra is given, a theorem about relative M\"{u}ger center in \cite{LC} is reviewed. Using central extension of Hopf algebras and the theorem \cite[Theorem 4.3]{LC}, we establish Theorems \ref{thm3.rf}-\ref{thm3.4} as main ways to construct ribbon factorizable Hopf algebras in Section 3. In Section 4, we apply the proposed methods to construct two families of pointed ribbon factorizable Hopf algebras $H(\omega,n),K(\alpha,n)$, and describe them by using generator and relations, prove that their representation categories are prime modular tensor categories which are not tensor equivalent to representation categories of small quantum groups or the ones in \cite[Example 5.18]{LC}. Then a family of semisimple ribbon factorizable Hopf algebras $A(p,q)$ are constructed. And we decompose $\Rep(A(p,q))$ into prime modular tensor categories. Although $\Rep(A(p,q))$ is not prime, we prove that $A(p,q)$ can't be obtained by some obvious way, i.e. it's not tensor product of trivial Hopf algebras (group algebras or their dual) and Drinfeld doubles.
\begin{convention}\emph{Throughout this paper we work over an algebraically closed field $\Bbbk$ of characteristic 0. All Hopf algebras in this paper are finite-dimensional. For the symbol $\delta$, we mean the classical Kronecker's symbol. Our references for the theory of Hopf algebras are \cite{R,MS}. For a Hopf algebra $H$, the set of group-like elements in $H$ will be denoted by $G(H)$.}
\end{convention}

\section{Preliminaries}
We collect some necessary notions and results in this section.
\subsection{Modular tensor categories and ribbon factorizable Hopf algebras.}
A finite tensor category $(\mathcal{C}, \otimes, \mathrm{1})$ \cite{ENO} is a rigid monoidal category such that $\mathcal{C}$ is a finite abelian category, the tensor product of $\mathcal{C}$ is $\Bbbk$-linear in each variable, and the unit object of $\mathcal{C}$ is a simple object. If $(\mathcal{C}, \otimes, \mathrm{1})$ admits a braiding $c$, then $(\mathcal{C}, \otimes, \mathrm{1}, c)$ is called braided finite tensor category (\cite{ENO}). In this case, the M\"{u}ger center $\mathcal{C}'$ is the full subcategory on the objects $\Ob(\mathcal{C}')=\{X \in \mathcal{C}|\; c_{Y,X} c_{X,Y}=\Id_{X\otimes Y} \text{ for all }Y\in \mathcal{C}\} $. Denote the tensor category of finite-dimensional vector spaces over $\Bbbk$ as $\mathbf{vect}_{\Bbbk}$.

\begin{definition}\cite[Theorem 1]{Shi}
A braided finite tensor category $(\mathcal{C}, \otimes, \mathrm{1}, c)$ is called
non-degenerate if the M\"{u}ger center $\mathcal{C}'$ is equal to $\mathbf{vect}_{\Bbbk}$.
\end{definition}

To give the definition of modular tensor categories, we need to recall ribbon tensor category as follows. A braided tensor category $(\mathcal{C}, \otimes, \mathrm{1},c)$ is ribbon if it is equipped with a natural isomorphism $\theta_X : X \xrightarrow{\sim} X$ (a twist) satisfying $\theta_{X\otimes Y}=(\theta_X \otimes \theta_Y)\circ c_{Y,X} \circ c_{X,Y}$ and $(\theta_X)^\ast=\theta_{X^\ast} $ for all $X, Y\in \mathcal{C}$. Now we review the following definition.

\begin{definition}\cite[Definition 5.2.7]{Ker}, \cite[Theorem 1]{Shi} \label{def2.3}
A braided finite tensor category is called modular if it is ribbon and non-degenerate.
\end{definition}

Let's introduce the following result which is implied in the proof of \cite[Theorem 4.3]{LC}. To do this, we need more concepts. Recall that a full subcategory of an abelian category is called \emph{topologizing subcategory} \cite{RA, Shi} if it is closed under finite direct sums and subquotients. By a \emph{braided tensor subcategory} of a braided tensor category $(\mathcal{C}, \otimes, \mathrm{1}, c)$ we mean a subcategory of $\mathcal{C}$ containing the unit object of $\mathcal{C}$, closed under the tensor product of $\mathcal{C}$, and containing the braiding isomorphisms. Let $S$ be a subset of objects of a braided category $(\mathcal{C}, \otimes, \mathrm{1},c)$, the M\"{u}ger centralizer $C_\mathcal{C}(S)$ \cite[Definition 2.6]{Mug} of $S$ in $\mathcal{C}$ is defined as the full subcategory of $\mathcal{C}$ with objects
$$\Ob(C_\mathcal{C}(S)) := \{X \in \mathcal{C} | \;c_{Y,X} c_{X,Y}=\Id_{X\otimes Y} \text{ for all } Y\in S\}.$$
The following theorem is implied by the proof of \cite[Theorem 4.3]{LC}.
\begin{theorem}\label{thm2.2}
Let $(\mathcal{D}, \otimes, \mathrm{1}, c)$ be a non-degenerate finite braided tensor category, let $\mathcal{E}$ be a topologizing braided tensor subcategory of $\mathcal{D}$. Then $C_\mathcal{D}(\mathcal{E})$ is non-degenerate braided tensor category if and only if $\mathcal{E}$ is non-degenerate.
\end{theorem}
Now we recall general prime modular tensor category. 
\begin{definition}\cite{Mug,LC}
A modular tensor category $\mathcal{C}$ is \emph{prime} if every topologizing non-degenerate braided tensor subcategory is equivalent to either $\mathcal{C}$ or $\mathbf{vect}_{\Bbbk}$.
\end{definition}
A fact that every modular tensor category is equivalent to a finite Deligne tensor product of prime modular categories is shown in \cite[Corollary 4.20]{LC}. Therefore, construction of modular tensor categories reduces to prime ones.

To produce modular tensor categories (MTCs) with Hopf algebras, we need definition of ribbon factorizable Hopf algebras. Recall that a quasitriangular Hopf algebra is a pair $(H, R)$ where $H$ is a Hopf algebra over $\Bbbk$ and $R=\sum R^{(1)} \otimes R^{(2)}$ is an invertible element in $H\otimes H$ such that
\begin{equation*}
 (\Delta \otimes \id)(R)=R_{13}R_{23},\; (\id \otimes \Delta)(R)=R_{13}R_{12},\;\Delta^{op}(h)R=R\Delta(h),
 \end{equation*}
for $h\in H$. Here by definition $R_{12}= \sum R^{(1)} \otimes R^{(2)}\otimes 1 $ and similarly for $R_{13}$ and $R_{23}$. Next, we recall ribbon structure on a quasitriangular Hopf algebra $(H,R)$. Let $u=\sum_i S(R^i)R_i$ and let $c=uS(u)$. If there is a center element $v\in H$ satisfying the following conditions:
\begin{itemize}
  \item[(R.1)] $v^2=c$;
  \item[(R.2)]  $S(v)=v$;
  \item[(R.3)] $\epsilon(v)=1$;
  \item[(R.4)] $\Delta(v)=(v\otimes v)(R_{21}R)^{-1}$;
\end{itemize}
then $v$ is called a ribbon element of $(H,R)$ and the triple $(H,R,v)$ is said to be a ribbon Hopf algebra. Let $\text{R}(H,R)$ denote the set of ribbon elements of $(H, R)$. Let $\mathrm{g}=uS(u^{-1})$. It's known that $S^4(h)=\mathrm{g}h\mathrm{g}^{-1}$ for $h\in H$ (see \cite[Proposition 12.3.2]{R} for example), and we will use this property frequently in next section. Define the set $G(\text{R}(H,R))$ as follows:
$$G(\text{R}(H,R))=\{l\in G(H)|\;l^2=\mathrm{g}^{-1}\text{ and } S^2(h)=l^{-1}hl \text{ for  all } h\in H\}.$$

The following theorem establishes a one-to-one correspondence between $G(\text{R}(H,R))$ and $\text{R}(H,R)$.

\begin{theorem}\cite[Theorem 1]{Kau}\label{thm2.r}
The map $G(\emph{R}(H,R))\rightarrow \emph{R}(H,R)$ given by $l\mapsto ul$ is a set of bijection.
\end{theorem}

For a quasitriangular Hopf algebra $(H,R)$, there are linear maps $f_{R_{21}R}: H^{\ast cop}\rightarrow H$ and $g_{R_{21}R}: H^{\ast op}\rightarrow H$, given respectively by
\begin{equation}\label{eq2.5}
f_{R_{21}R}(a):=(a\otimes \Id)(R_{21}R),\;\;g_{R_{21}R}(a):=(\Id\otimes a)(R_{21}R),\;f\in H^\ast.
\end{equation}
A factorizable Hopf algebra is a quasitriangular Hopf algebra $(H, R)$ such that $f_{R_{21}R}$, or equivalently $g_{R_{21}R}$, is a linear isomorphism. If a quasitriangular Hopf algebra is ribbon and factorizable, then we call it a \emph{ribbon factorizable} Hopf algebra. Let $(H,R)$ be a quasitriangular Hopf algebra and let $(\Rep(H), \otimes, \Bbbk, c)$ be the finite braided tensor category of finite-dimensional representations of $H$, where $c_{X,Y}$ is given by $\tau\circ R$ for $X,Y\in \Rep(H)$, where $\tau$ is the flip map. It's known that $\Rep(H)$ is non-degenerate braided finite tensor category if and only if $(H, R)$ is a factorizable Hopf algebra. Moreover, it is modular if and only if $(H, R)$ is a ribbon factorizable Hopf algebra (see \cite[Section 2.5]{Shi1} for example).

A well known class of factorizable Hopf algebras are Drinfel'd doubles and we recall them as follows. Let $H$ be a Hopf algebra over $\Bbbk$. The Drinfel'd double of $H$ denoted by $D(H)$ is given as follows. As a coalgebra, $D(H)=(H^*)^{cop}\otimes H$. The multiplication of $D(H)$ is given by $(f\otimes h)(g\otimes k)=f [h_{(1)}\rightharpoonup g\leftharpoonup S^{-1}(h_{(3)})]\otimes h_{(2)}k$, where $f,g\in H^*, \;h,k\in H$ and $\langle a\rightharpoonup g\leftharpoonup b,c\rangle:=\langle g,bca\rangle$ for $a,b,c\in H$. For convenience, we write $fh$ as $f\otimes h$ in the following content. Let $\mathcal{R}$ be the standard universal $\mathcal{R}$-matrix of $D(H)$, i.e. $\mathcal{R}=\sum_{i=1}^n h_i\otimes h^i$, where $n=\
\dim(H)$ and $\{h_i\}_{i=1}^n$, $\{h^i\}_{i=1}^n$ are dual basis of $H$.

\subsection{Hopf exact sequence.}
\begin{definition}\label{def2.1.1}
A short exact sequence of Hopf algebras is a sequence of Hopf algebras
and Hopf algebra maps
\begin{equation}\label{ext}
\;\; K\xrightarrow{\iota} H \xrightarrow{\pi} \overline{H}
\end{equation}
such that
\begin{itemize}
  \item[(i)] $\iota$ is injective,
  \item[(ii)]  $\pi$ is surjective,
  \item[(iii)] $\ker(\pi)= HK^+$, $K^+$ is the kernel of the counit of $K$,
\end{itemize}
\end{definition}
here we view $K$ as a sub-Hopf algebra of $H$ through the map $\iota$. Take an exact sequence \eqref{ext}, then $K$ is a normal Hopf
subalgebra of $H$. Conversely, if $K$ is a normal Hopf subalgebra of a Hopf algebra
$H$, then the quotient coalgebra $H=H/HK^+=H/K^+H$ is a quotient Hopf algebra
and $H$ fits into an extension \eqref{ext}, where $\iota$ and $\pi$ are the canonical maps. If $H$ fits into an extension \eqref{ext} and $H$ is finite-dimensional, then $\dim(H)=\dim(K)\dim(\overline{H})$ (see \cite{Sch}).

An extension \eqref{ext} above is called \emph{abelian} if $K$ is commutative and $\overline{H}$ is cocommutative. Abelian extensions were classified by Masuoka
(see \cite[Proposition 1.5]{M3}), which can be expressed as $\Bbbk^G\#_{\sigma,\tau}\Bbbk F$, where $G,F$ are finite groups.
To give the description of $\Bbbk^G\#_{\sigma,\tau}\Bbbk F$, we need the following data
\begin{itemize}
\item[(i)] A matched pair of groups, i.e. a quadruple $(F,G,\triangleleft,\triangleright)$, where $G\stackrel{\triangleleft}{\leftarrow}G\times F \stackrel{\triangleright }{\rightarrow}F$ are action of groups on sets, satisfying the following conditions
    \begin{align*}
    g\triangleright(ff')=(g\triangleright f)((g\triangleleft f)\triangleright f'),\quad (gg')\triangleleft f=(g\triangleleft(g'\triangleright f))(g'\triangleleft f),
    \end{align*}
    for $g,g'\in G$ and $f,f'\in F$.
\item[(ii)] $\sigma:G\times F\times F\rightarrow \Bbbk^\times$ is a map such that
\begin{align*}
    \sigma(g\triangleleft f,f',f'')\sigma(g,f,f'f'')=\sigma(g ,f,f')\sigma(g,ff',f'')
    \end{align*}
    and $\sigma(1,f,f')=\sigma(g,1,f')=\sigma(g,f,1)=1$, for $g\in G$ and $f,f',f''\in F$.
\item[(iii)] $\tau:G\times G \times F \rightarrow \Bbbk^\times$ is a map satisfying
    \begin{align*}
    \tau(gg',g'',f)\tau(g,g',g''\triangleright f)=\tau(g',g'',f)\tau(g,g'g'',f)
    \end{align*}
    and $\tau(g,g',1)=\tau(g,1,f)=\tau(1,g',f)$, for $g,g',g''\in G$ and $f\in F$. Moreover, the $\sigma, \tau$ satisfy the following compatible condition
    \begin{align*}
    \sigma(gg',f,f')\tau(g,g',ff')&=\sigma(g,g'\triangleright f, (g'\triangleleft f)\triangleright f')\sigma(g',f,f')\\
    &\tau(g,g',f)\tau(g\triangleleft (g'\triangleright f),g'\triangleleft f,f'),
    \end{align*}
    for $g,g',g''\in G$ and $f,f',f''\in F$.
\end{itemize}

\begin{definition}\cite[Section 2.2]{AA}\label{def2.1.2}
The Hopf algebra $\Bbbk^G\#_{\sigma,\tau}\Bbbk F$ is equal to $\Bbbk^G\otimes \Bbbk F$ as vector space and we write $a\otimes x$ as $a\#x$. The product, coproduct are given by
\begin{align*}
 &(e_g\#f).(e_{g'}\#f')=\delta_{g\triangleleft f,g'}\;\sigma(g,f,f')\;e_g\#(ff'),\\
 &\Delta(e_g\#f)=\sum_{g'g''=g}\;\tau(g',g'',f)\;e_{g'}\#g''\triangleright f\otimes e_{g''}\#f,
\end{align*}
The unit is $\sum_{g\in G}e_g\#1$ and the counit is $\epsilon(e_g\#f)=\delta_{g,1}$ and the antipode is
\begin{align*}
 S(e_g\#f)=\sigma(g^{-1},g\triangleright f,(g\triangleright f)^{-1})^{-1}\;
 \tau(g^{-1},g,f)^{-1}\;e_{(g\triangleleft f)^{-1}}\#(g\triangleright f)^{-1}.
\end{align*}
\end{definition}
Using above definition, we review Hopf algebras $\mathscr{A}_l$ which will be used to construct semisimple ribbon factorizable Hopf algebras.
\begin{example}\label{ex2.1.5}
\emph{Let $p, q$ be two odd prime numbers such that $p\equiv 1(\text{mod}\;q)$ and let $\omega$ be a primitive $q$th root of 1 in $\Bbbk$. Assume $t\in \mathbb{N}$ satisfying $t^q\equiv 1(\text{mod}\;p)$ and $t \not \equiv 1(\text{mod}\;p)$. Let $0\leq l\leq (q-1)$, then the Hopf algebra $\mathscr{A}_l$ \cite[Lemma 1.3.9]{Na2} belongs to $\Bbbk^G\#_{\sigma,\tau}\Bbbk F$. By definition, the data $(G,F,\triangleleft,\triangleright,\sigma,\tau)$ of $\mathscr{A}_l$ is given by the following way
\begin{itemize}
             \item[(i)] $G=\mathbb{Z}_p\rtimes \mathbb{Z}_q=\langle a,b|\;a^p=b^q=1,bab^{-1}=a^t\rangle$,\;$F=\mathbb{Z}_q=\langle x|\;x^q=1\rangle$. The action $\triangleright $ is trivial, and $a\triangleleft x^{i}=a^{t^i},b\triangleleft x^i=b$, for $0\leq i \leq q-1$.
              \item[(ii)] $\sigma(a^i b^j,x^m, x^n)=w^{jlq_{mn}}$, where $q_{mn}$ is the quotient of $m+n$ in the division by $q$ and $1 \leq i\leq p-1$, $0\leq j,m,n \leq q-1$.
              \item[(iii)] $\tau(g ,g',f)=1$ for $g,g'\in G$ and $f\in F$.
\end{itemize}}
\end{example}

\subsection{The pointed Hopf algebras $u(\mathcal{D},\lambda,\mu)$.}
Assume $H$ is a finite-dimensional pointed Hopf algebra and $G(H)$ is an abelian group. In \cite{AH}, N. Andruskiewitsch and H.J. Schneider have classified all these Hopf algebras such that all prime divisors of $|G(H)|$ are greater than 7. As a result, such a Hopf algebra is isomorphic to the Hopf algebra $u(\mathcal{D},\lambda,\mu)$ which we will recall, where $\mathcal{D},\lambda,\mu$ are certain parameters. Fix a finite abelian group $G$. Denote the dual group of $G$ as $\widehat{G}$. A datum $\mathcal{D}$ of finite Cartan type for $G$,
$$\mathcal{D}=\mathcal{D}(G, (g_i)_{1\leq i\leq \theta}, (\chi_i)_{1\leq i\leq \theta}, (a_{ij})_{1\leq i,j\leq \theta}),$$
consists of elements $g_i\in G$, $\chi_i\in \widehat{G}$ for $1\leq i \leq \theta$, and a Cartan matrix $(a_{ij})_{1\leq i,j\leq \theta}$ of finite type satisfying
\begin{equation}\label{eq2.c}
q_{ij}q_{ji}=q_{ii}^{a_{ij}},\;q_{ii}\neq 1, \;\text{with }q_{ij}=\chi_j(g_i),\;1\leq i,j \leq \theta.
\end{equation}

Let $\Phi$ be the root system of the Cartan matrix $(a_{ij})_{1\leq i,j\leq \theta}$, $\alpha_1,...,\alpha_\theta$ a system of simple roots, and $T$ the set of connected components of the Dynkin diagram of $\Phi$. Let $\Phi_J$ be the root system of the component $J$, where $J\in T$. We write $i\sim j$ if $\alpha_i$ and $\alpha_j$ are in the same connected component of the Dynkin diagram of $\Phi$.

For a positive root $\alpha=\sum_{i=1}^\theta n_i \alpha_i$, where $n_i\in \mathbb{N}$ for $1\leq i \leq \theta$, define
$$g_\alpha=\prod_{i=1}^\theta g_i^{n_i},\;\chi_\alpha=\prod_{i=1}^\theta \chi_i^{n_i}.$$
We assume that the order of $q_{ii}$ is odd for $1\leq i \leq \theta$, and that the order of $q_{ii}$ is prime to $3$ for all $i$ in a connected component of type $G_2$. By the equations \eqref{eq2.c}, the order $N_i$ of $q_{ii}$ is constant in each connected component $J$, and we define $N_J =N_i$ for all $i\in J$.

\emph{The parameter} $\lambda$. Let $\lambda=(\lambda_{ij})_{1\leq i ,j\leq \theta}$ be a family of elements in $\Bbbk$ satisfying the following condition for all $1\leq i<j \leq \theta$: if $g_ig_j=1$ or $\chi_i\chi_j\neq 1$, then $\lambda_{ij}=0$. Moreover, $\lambda_{ji}=-q_{ji}\lambda_{ij}$.

\emph{The parameter} $\mu$. Let $\mu=(\mu_\alpha)_{\alpha\in \Phi}$ be a family of elements in $\Bbbk$ such that for all $\alpha\in \Phi_J$, $J\in T$, if $g_\alpha=1$ or $\chi_\alpha^{N_J}\neq1$, then $\mu_\alpha=0$.

\emph{The Hopf algebra} $u(\mathcal{D},\lambda,\mu)$. Given a parameter $\mu$ for a datum $\mathcal{D}$. Let $\alpha\in \Phi$. In \cite{AH}, the authors associate to $\mu$ and $\alpha$ an element $u_\alpha(\mu)$ in the group algebra $\Bbbk G$ (see \cite[Section 2.2]{AH} for details). By their construction, $u_\alpha(\mu)$ belongs to the augmentation ideal of $\Bbbk[g_i^{N_i}|\;1\leq i\leq \theta]$, and $u_{\alpha_i}(\mu)=u_{\alpha_i}(1-g_i^{N_i})$ for all simple roots $\alpha_i$, $1\leq i\leq \theta$.

The Hopf algebra $u(\mathcal{D},\lambda,\mu)$ is generated as an algebra by the group $G$, that is, by generators of $G$ satisfying the relations of the group, and $x_1,...,x_\theta$, with the relations:
\begin{align}
\label{eq2.1.x} gx_ig^{-1}=\chi_i(g)x_i,\; for\;  1\leq i\leq \theta,\; g\in G;\\
\label{eq2.2.x} \text{ad}_{c} (x_i)^{1-a_{ij}}(x_j)=0,\; for\;  i\neq j,i\sim j;\\
\label{eq2.3.x} \text{ad}_{c} (x_i)(x_j)=\lambda_{ij}(1-g_ig_j),\; for\;  i< j,i\not \sim j;\\
\label{eq2.4.x} x_\alpha^{N_J}=u_\alpha(\mu),\; for\;  \alpha\in \Phi_J^{+},J\in T;
\end{align}
where $\text{ad}_{c} (x_i)$ is the braided adjoint action of $x_i$, and the $x_\alpha$ is the root vector associated to $\alpha$ (see \cite[Sections 1.2, 2.1]{AH} for details). The coalgebra structure is given by
\begin{align*}
\Delta(x_i)=g_i\otimes x_i+x_i\otimes 1,\;\Delta(g)=g\otimes g,\; for\;  1\leq i\leq \theta,\; g\in G.
\end{align*}
Suppose $\Phi^+=\{\beta_1,...,\beta_p\}$ is a convex ordering. Then the following result is shown in \cite[Theorem 4.5]{AH}.

\begin{theorem}\label{thm2.3}
The set $\{x_{\beta_1}^{i_1}...x_{\beta_p}^{i_p}g|\;0\leq i_k\leq N_k,\;1\leq k\leq p,\;g\in G\}$ is a linear basis of $u(\mathcal{D},\lambda,\mu)$. In particular,
$$\dim(u(\mathcal{D},\lambda,\mu))=\prod_{J\in T}N_J^{|\Phi_J^+|}|G|.$$
\end{theorem}

Next, we review the Hopf algebra $U(\mathcal{D},\lambda)$ given in \cite[Definition 3.2]{AH}. The Hopf algebra $U(\mathcal{D},\lambda)$ is generated as an algebra by the group $G$ and $\{X_1,...,X_\theta\}$ subject to the relations
\begin{align*}
gX_ig^{-1}=\chi_i(g)X_i,\; for\;  1\leq i\leq \theta,\; g\in G;\\
\text{ad}_{c} (X_i)^{1-a_{ij}}(X_j)=0,\; for\;  i\neq j,i\sim j;\\
\text{ad}_{c} (X_i)(X_j)=\lambda_{ij}(1-g_ig_j),\; for\;  i< j,i\not \sim j;
\end{align*}
The coalgebra structure of $U(\mathcal{D},\lambda)$ is given by
\begin{align*}
\Delta(X_i)=g_i\otimes X_i+X_i\otimes 1,\;\Delta(g)=g\otimes g,\; for\;  1\leq i\leq \theta,\; g\in G.
\end{align*}
Then $u(\mathcal{D},\lambda,\mu)=U(\mathcal{D},\lambda)/\langle X_\alpha^{N_J}-u_\alpha(\mu)|\;\alpha\in \Phi_J^+,J\in T\rangle$ as Hopf algebras (\cite{AH}). The following result is shown in \cite[Theorem 3.3]{AH}.
\begin{theorem}\label{thm2.4}
The set $\{X_{\beta_1}^{i_1}...X_{\beta_p}^{i_p}g|\;i_k\geq 0,\;1\leq k\leq p,\;g\in G\}$ is a linear basis of $U(\mathcal{D},\lambda)$.
\end{theorem}


%

\subsection{Nichols algebras and the Hopf algebras $\mathcal{U}(\mathcal{D}_{red})$.}
In this subsection, some results about Nichols algebras of diagonal type will be reviewed, and then we will recall the Hopf algebras $\mathcal{U}(\mathcal{D}_{red})$. And one can refer to \cite{ADH,AN} for more details.

\emph{Yetter-Drinfeld category over finite abelian groups.}  Let $G$ be an abelian group. A Yetter-Drinfeld module over the group algebra $\Bbbk G$ can be described as a $G$-graded vector space which is a $G$-module such that all $g$-homogeneous components are stable under the $G$-action, where $g\in G$. Denote the category
of Yetter-Drinfeld modules over $\Bbbk G$ by $_G^G\mathcal{YD}$.

For $V \in _G^G\mathcal{YD}$, $g\in G$, and $\chi\in \widehat{G}$ we define $V_g = \{v\in V | \;\rho(v)=g\otimes v\}$ and $V_g^\chi=\{v\in V_g| \;h.v=\chi(h)v \text{ for all } h\in G \}$. The category $_G^G\mathcal{YD}$ is braided, and for $V,W\in _G^G\mathcal{YD}$
$$c:V\otimes W \rightarrow W\otimes V, v\otimes w \mapsto g.w \otimes v, v\in V_g, w \in W,$$
defines the braiding on $V\otimes W$.

Let $g_1,...,g_\theta \in G, \chi_1,...,\chi_\theta \in G, \theta\in \mathbb{N}^+$. Let $V\in _G^G\mathcal{YD}$ with basis $v_i\in V_{g_i}^{\chi_i},1\leq i \leq \theta$. Then $V$ is a braided vector space with
$$c(v_i\otimes v_j)=\chi_j(g_i)v_j\otimes v_i, \quad 1\leq  i, j \leq \theta.$$
The matrix $(q_{ij})_{1\leq  i, j \leq \theta}$ is called the braiding matrix of
the elements $(v_i)$, where $q_{ij}=\chi_j(g_i)$ for all $1\leq  i, j \leq \theta$. If $q_{ii}\neq 1$ and there are integers $a_{ij}$ with $a_{ii}=2$, $1\leq  i, j \leq \theta$, and $0\leq -a_{ij} < \text{ord}q_{ii}$ (which could be infinite) for $1\leq  i\neq j \leq \theta$, such that
$$q_{ij}q_{ji}=q_{ii}^{a_{ij}},\quad 1\leq  i, j \leq \theta,$$
then $(V,c)$ is called \emph{Cartan type} in \cite[Definition 4.4]{AH1}. Recall the following braided vector spaces of rank $2$ which are Cartan type (see for example \cite[Theorem 7.1]{HI}).

\begin{example}\label{ex3.1.x}
\emph{Let $n\geq 2$ and let $G=\mathbb{Z}_{2n}\times \mathbb{Z}_{2n}=\langle a,b|\;a^{2n}=b^{2n}=1,ab=ba\rangle$. Let $\omega$ be a primitive $2n$-th root of $1$. Define $g_1=a,g_2=b$ and $\chi_1,\chi_2\in \widehat{G}$ by
\begin{align*}
\chi_1(a)=\omega^{-3},\chi_1(b)=\omega^2\\
\chi_2(a)=\omega,\chi_2(b)=\omega^{-3}.
\end{align*}
Let $U_2=\oplus_{i=1}^2\Bbbk u_i\in _G^G\mathcal{YD}$ with basis $u_i\in (U_2)_{g_i}^{\chi_i}$, $1\leq i\leq 2$. Then the braiding matrix $(q_{ij})_{1\leq i,j\leq 2}$ of $(u_i)_{i=1}^2$ is
$$\begin{pmatrix} \omega^{-3} & \omega \\ \omega^2 & \omega^{-3} \end{pmatrix}.$$
By definition, $(U_2,c)$ is type $A_2$.}
\end{example}

Next, we recall a family of braided vector spaces of rank $2$ which are not of Cartan type (see for example \cite[Section 3.3]{AH1}).

\begin{example}\label{ex3.1}
\emph{Let $n$ be an odd number and let $G=\mathbb{Z}_{3n}\times \mathbb{Z}_{3n}=\langle a,b|\;a^{3n}=b^{3n}=1,ab=ba\rangle$. Let $\alpha$ be a primitive $3n$-th root of $1$. Define $g_1=a,g_2=b$ and $\chi_1,\chi_2\in \widehat{G}$ by
\begin{align*}
\chi_1(a)=\alpha^{-2},\chi_1(b)=\alpha\\
\chi_2(a)=\alpha,\chi_2(b)=\alpha^{n}.
\end{align*}
Let $V_2=\oplus_{i=1}^2\Bbbk v_i\in _G^G\mathcal{YD}$ with basis $v_i\in (V_2)_{g_i}^{\chi_i}$, $1\leq i\leq 2$. Then the braiding matrix $(q_{ij})_{1\leq i,j\leq 2}$ of $(v_i)_{i=1}^2$ is
$$\begin{pmatrix} \alpha^{-2} & \alpha \\ \alpha & \alpha^{n} \end{pmatrix}.$$}
\end{example}

Let $V\in _G^G\mathcal{YD}$ with basis $v_i\in V_{g_i}^{\chi_i},1\leq i \leq \theta$. The tensor algebra $T(V)$ is a braided Hopf algebra in $_G^G\mathcal{YD}$ where the elements in $V$ are primitive. We identify $T(V)$ with the free algebra $\Bbbk\langle v_1,...,v_\theta\rangle$. Recall that the $\Bbbk$-linear endomorphism $\text{ad}_c(v_i)$ of the free algebra is given for all $y = x_{i1}...x_{it}$ by
the braided commutator
$$\text{ad}_c(v_i)(y) = v_i.y-q_{ii_1}...q_{ii_t}y.v_i.$$
The smash product $T(V)\# \Bbbk G$ is a usual Hopf algebra (also called a bosonization of $T(V)$) with
$$gv_ig^{-1}=\chi_i(g)v_i,\;\Delta(v_i)=g_i\otimes v_i+v_i\otimes 1, \; \Delta(g)=g\otimes g .$$
for all $1\leq i \leq \theta,g\in G$.

\emph{Nichols algebra of $V$.} Let $V \in _G^G\mathcal{YD}$. A braided graded Hopf algebra
$$R=\oplus_{n\geq 0}R(n) \quad \text{in }  _G^G\mathcal{YD}$$
is a Nichols algebra of $V$ if $V \cong R(1)$ in $_G^G\mathcal{YD}$, and if
\begin{itemize}
\item[(i)] $R(0)=\Bbbk 1$;
\item[(ii)] $R(1)=P(R)$, where $P(R)=\{x\in R|\;\Delta(x)=x\otimes 1+ 1\otimes x\}$ is the set of primitive elements of $R$;
\item[(iii)] $\Bbbk \langle R(1)\rangle=R$, the $\Bbbk \langle R(1)\rangle$ denotes the subalgebra generated by $R(1)$.
\end{itemize}
The Nichols algebra of $V$ exists and up to isomorphism it is uniquely determined by $V$ (see for example \cite[Section 3]{AH1}), which can be given by $T(V)/\mathcal{J}(V)$, here $\mathcal{J}(V)$ is the largest graded coideal of the $T(V)$ such that $\mathcal{J}(V)\cap V=(0)$. Denote the Nichols algebra of $V$ by $\mathfrak{B}(V)$. As an algebra and coalgebra, $\mathfrak{B}(V)$ only depends on the braided vector space $(V, c)$. The Nichols algebra of $U_2$ given in Example \ref{ex3.1.x} can be represented as follows.

\begin{example}\label{ex3.2.x}
\emph{Let $(U_2,c)$ be the braided space in Example \ref{ex3.1.x}. Define
\begin{align*}
a=u_1,\;b=\text{ad}_c u_1(u_2),\;c=u_2.
\end{align*}
Then $\mathfrak{B}(U_2)$ is generated by $\{a,b,c\}$ as an algebra subject to
\begin{align}
\label{eq2.5.y}&a^{2n}=b^{2n}=c^{2n}=0,\\
\label{eq2.6.y}&ba=q_{11}^{-1}q_{12}^{-1}ab,\;ca=q_{12}^{-1}(ac-b),\;cb=q_{12}^{-1}q_{22}^{-1}bc,
\end{align}
and the set $\{a^i b^j c^k|\;0\leq i,j ,k < 2n\}$ is a linear basis of $\mathfrak{B}(U_2)$. Hence $\dim(\mathfrak{B}(U_2))=(2n)^3$.}
\end{example}

The Nichols algebra of $V_2$ given in Example \ref{ex3.1}, is called \emph{type $T_3$} in \cite{HI}. And it had been described explicitly in \cite[Theorem 7.1]{HI}.

\begin{example}\label{ex3.2}
\emph{Let $(V_2,c)$ be the braided space in Example \ref{ex3.1}. Define
\begin{align*}
a=v_1,\;b=\text{ad}_c v_1(v_2),\;c=\text{ad}_c^2 v_1(v_2),\;d=v_2.
\end{align*}
Then $\mathfrak{B}(V_2)$ is generated by $\{a,b,c,d\}$ as an algebra subject to
\begin{align}
\label{eq2.5.x}&a^3=b^3=c^{3n}=d^{3n}=0,\\
\label{eq2.6.x}&ba=q_{11}^{-1}q_{12}^{-1}(ab-c),\;ca=(2q_{11}q_{12})^{-1}ac,\;da=q_{12}^{-1}(ad-b),\\
\label{eq2.7.x}&cb=q_{11}^2q_{12}bc,\;db=(q_{12}q_{22})^{-1}bd,\;dc=q_{12}^{-2}q_{22}^{-1}[cd-\frac{\mu}{(2)_{\omega}! } b^2],
\end{align}
where $\omega=q_{22}^{-1},\mu=(1+q_{11}^{-1})q_{21}^{-1}-(1+q_{11})q_{12}$, and the set $\{a^i b^j c^k d^l|\;0\leq i,j < 2,\;0\leq k,l < 3n\}$ is a linear basis of $\mathfrak{B}(V_2)$. Hence $\dim(\mathfrak{B}(V_2))=3^4 n^2$.}
\end{example}

\emph{The Hopf algebras $\mathcal{U}(\mathcal{D}_{red})$.} Let $\mathfrak{q}=(q_{ij})_{1\leq i,j \leq \theta}$ be a matrix of elements in $\Bbbk$ such that $q_{ii}\neq1$ for all $1\leq i \leq \theta$. Let $G$ be a finite abelian group. A reduced YD-datum (see \cite[Definition 3.2]{ADH} for example) is a collection $\mathcal{D}_{red}=(f_i, g_i, \chi_i)_{1\leq i\leq \theta}$, where $f_i,g_i\in G,\chi_i\in \widehat{G}$, $1\leq i \leq \theta$ such that
$$q_{ij}=\chi_j(g_i)=\chi_i(f_j),\;f_ig_i\neq 1,\;1\leq i,j \leq \theta.$$
Now fix a reduced YD-datum as above and define
\begin{align*}
&V=\oplus_{i=1}^\theta \Bbbk v_i \in _G^G\mathcal{YD}, \quad \quad \text{with basis } v_i\in V_{g_i}^{\chi_i},1\leq i \leq \theta,\\
&W=\oplus_{i=1}^\theta \Bbbk w_i \in _G^G\mathcal{YD}, \quad \;\; \text{with basis } w_i\in W_{f_i}^{\chi_i^{-1}},1\leq i \leq \theta.
\end{align*}
Let $\mathcal{U}(\mathcal{D}_{red})$ be the quotient of $T(V\oplus W)\# \Bbbk G$ by the ideal generated by the relations of the Nichols algebras $\mathcal{J}(V)$ and $\mathcal{J}(W)$, together with
$$v_iw_j-\chi_j^{-1}(g_i)w_jv_i-\delta_{i,j}(f_ig_i-1),\;1\leq i,j \leq \theta.$$
To express $\mathcal{U}(\mathcal{D}_{red})$ as a quotient of a Drinfeld double, we recall the result in \cite[Theorem 3.7]{ADH} as follows. Let $A=\mathfrak{B}(V)\# \Bbbk G$ and let $U=\mathfrak{B}(W)\# \Bbbk \widehat{G}$. By \cite[Theorem 3.7]{ADH}, there is a non-degenerate skew-Hopf bilinear form $\langle,\rangle:U\otimes A\rightarrow \Bbbk$ given by
\begin{align*}
\langle \chi, g\rangle=\chi^{-1}(g),\;\langle \chi, v_i\rangle=0;\\
\langle w_i, g\rangle=0,\;\langle w_i, v_j\rangle=-\delta_{ij},
\end{align*}
where $g\in G,\chi \in \widehat{G}$. The following result is special case of \cite[Theorem 4.11]{RH1}.
\begin{theorem}\label{thm2.5}
$\mathcal{U}(\mathcal{D}_{red})\cong D(A)/\langle \chi_if_i-1|\;1\leq i\leq \theta\rangle$ as Hopf algebras, where $\langle \chi_if_i-1|\;1\leq i\leq \theta\rangle$ is the ideal generated by $\{\chi_if_i-1|\;1\leq i\leq \theta\}$.
\end{theorem}
If $\mathfrak{B}(V)$ is finite dimensional, then we can assume $\mathcal{R}$ is the standard universal $\mathcal{R}$-matrix of $D(A)$. By the theorem above, there is a universal $\mathcal{R}$-matrix of $\mathcal{U}(\mathcal{D}_{red})$ through the quotient map and we denote it as $\overline{\mathcal{R}}$ in the following content.

By definition, $\mathcal{U}(\mathcal{D}_{red})$ is generated as an algebra by the group $G$, and $\{v_i,w_i|\;1\leq i\leq \theta\}$, with the relations:
\begin{align}
\label{eq2.7} & gv_ig^{-1}=\chi_i(g)v_i,\;gw_ig^{-1}=\chi_i^{-1}(g)v_i,\; for\;  1\leq i\leq \theta,\; g\in G;\\
\label{eq2.8} &v_iw_j-\chi_j^{-1}(g_i)w_jv_i-\delta_{i,j}(f_ig_i-1),\;1\leq i,j \leq \theta;\\
\label{eq2.9} &\mathcal{J}(V),\mathcal{J}(W).
\end{align}
The coalgebra structure is given by
\begin{align*}
\Delta(g)=g\otimes g,\;\Delta(v_i)=g_i\otimes v_i+v_i\otimes 1,\; \Delta(w_i)=f_i\otimes w_i+w_i\otimes 1,
\end{align*}
for all $1\leq i\leq \theta,\; g\in G$.
Next, we give following examples which will be used in next sections.
\begin{example}\label{ex3.3.x}
\emph{We keep the notation in Example \ref{ex3.1.x}. To define a reduced YD-datum, we assume that $(n,7)=1$. Since $(n,7)=1$, we can find $n_0\in \mathbb{N}$ such that $7n_0\equiv 1(\text{mod }2n)$. To find a reduced YD-datum $\mathcal{D}_{red}=(g_i, f_i, \chi_i)_{1\leq i\leq 2}$, where $g_i,\chi_i$ are given in Example \ref{ex3.1.x} for $1\leq i \leq 2$, we need to solve the following equations
$$\chi_i(f_j)=\chi_j(g_i),\;1\leq i,j\leq 2.$$
In fact, there is one solution which is given by
$$f_1=g_1^{5n_0}g_2^{-3n_0}, f_2=g_1^{3n_0}g_2^{8n_0}.$$
Fix this reduced YD-datum $\mathcal{D}_{red}=(g_i, f_i, \chi_i)_{1\leq i\leq 2}$, $1\leq i \leq 2$. Since above discussion, the Hopf algebra $\mathcal{U}(\mathcal{D}_{red})$ is generated as an algebra $\{g,u_i,s_i|\;1\leq i\leq 2,\;g\in G\}$ subject to
\begin{align}
\label{eq2.10.y} & gv_ig^{-1}=\chi_i(g)v_i,\;gw_ig^{-1}=\chi_i^{-1}(g)v_i,\; for\;  1\leq i\leq 2,\; g\in G;\\
\label{eq2.11.y} &v_iw_j-\chi_j^{-1}(g_i)w_jv_i=\delta_{i,j}(f_ig_i-1),\;1\leq i,j \leq 2;
\end{align}
and the relations \eqref{eq2.5.y}-\eqref{eq2.6.y} in Example \ref{ex3.2.x}, and
\begin{align}
&\overline{a}^{2n}=\overline{b}^{2n}=\overline{c}^{2n}=0,\\
&\overline{b}\overline{a}=q_{11}q_{12}\overline{a}\overline{b},\;\overline{c}\overline{a}=q_{12}(\overline{a}\overline{c}-\overline{b}),\;\overline{c}\overline{b}=q_{12}q_{22}\overline{b}\overline{c},
\end{align}
where $\overline{a},\overline{b},\overline{c}$ are given by
\begin{align*}
\overline{a}=s_1,\;\overline{b}=\text{ad}_c s_1(s_2),\;\overline{c}=s_2.
\end{align*}
The set $\{\overline{a}^p \overline{b}^q \overline{c}^r a^i b^j c^kg|\;0\leq p,q,r,i,j,k< 2n,g\in G\}$ is a linear basis of $\mathcal{U}(\mathcal{D}_{red})$. Hence $\dim(\mathcal{U}(\mathcal{D}_{red}))=(2n)^8$. The coalgebra structure is given by
\begin{align*}
\Delta(g)=g\otimes g,\;\Delta(u_i)=g_i\otimes u_i+u_i\otimes 1,\; \Delta(s_i)=f_i\otimes s_i+s_i\otimes 1,
\end{align*}
for all $1\leq i\leq \theta,\; g\in G$.}
\end{example}

\begin{example}\label{ex3.3}
\emph{Using the notation in Example \ref{ex3.1} and define a reduced YD-datum by $\mathcal{D}_{red}=(g_i, g_i, \chi_i)_{1\leq i\leq 2}$, where $g_i\in G,\chi_i\in \widehat{G}$, $1\leq i \leq 2$. Since above discussion, the Hopf algebra $\mathcal{U}(\mathcal{D}_{red})$ is generated as an algebra $\{g,v_i,w_i|\;1\leq i\leq 2,\;g\in G\}$ subject to
\begin{align}
\label{eq2.10} & gv_ig^{-1}=\chi_i(g)v_i,\;gw_ig^{-1}=\chi_i^{-1}(g)v_i,\; for\;  1\leq i\leq 2,\; g\in G;\\
\label{eq2.11} &v_iw_j-\chi_j^{-1}(g_i)w_jv_i=\delta_{i,j}(g_i^2-1),\;1\leq i,j \leq 2;
\end{align}
and the relations \eqref{eq2.5.x}-\eqref{eq2.7.x} in Example \ref{ex3.2}, and
\begin{align}
&\overline{a}^3=\overline{b}^3=\overline{c}^{3n}=\overline{d}^{3n}=0,\\
&\overline{b}\overline{a}=q_{11}q_{12}(\overline{a}\overline{b}-\overline{c}),\;\overline{c}\overline{a}=(2q_{11}q_{12})\overline{a}\overline{c},\;\overline{d}\overline{a}=q_{12}(\overline{a}\overline{d}-\overline{b}),\\
&\overline{c}\overline{b}=q_{11}^{-2}q_{12}^{-1}\overline{b}\overline{c},\;\overline{d}\overline{b}=(q_{12}q_{22})\overline{b}\overline{d},\;\overline{d}\overline{c}=q_{12}^{2}q_{22}[\overline{c}\overline{d}-\frac{\overline{\mu}}{(2)_{\overline{\omega}}! } \overline{b}^2],
\end{align}
where $\overline{\omega}=q_{22},\mu=(1+q_{11})q_{21}-(1+q_{11}^{-1})q_{12}^{-1}$ and
\begin{align*}
\overline{a}=w_1,\;\overline{b}=\text{ad}_{c} w_1(w_2),\;\overline{c}=\text{ad}_{c}^2 w_1(w_2),\;\overline{d}=w_2.
\end{align*}
The set $\{\overline{a}^p \overline{b}^q \overline{c}^r \overline{d}^s a^i b^j c^k d^lg|\;0\leq p,q,i,j < 2,\;0\leq r,s,k,l < 3n,g\in G\}$ is a linear basis of $\mathcal{U}(\mathcal{D}_{red})$. Hence $\dim(\mathcal{U}(\mathcal{D}_{red}))=3^{10} n^6$. The coalgebra structure is given by
\begin{align*}
\Delta(g)=g\otimes g,\;\Delta(v_i)=g_i\otimes x_i+v_i\otimes 1,\; \Delta(w_i)=g_i\otimes w_i+w_i\otimes 1,
\end{align*}
for all $1\leq i\leq \theta,\; g\in G$.}
\end{example}

\begin{remark}
\emph{For simplicity, we denote the Hopf algebra $\mathcal{U}(\mathcal{D}_{red})$ in Example \ref{ex3.3.x} (resp. the Hopf algebra $\mathcal{U}(\mathcal{D}_{red})$ in Example \ref{ex3.3}) as $H(\omega,n)$ (resp. $K(\alpha,n)$). With some restrictions to their parameters, we will prove that $\Rep H(\omega,n)$ and $\Rep K(\alpha,n)$ are prime modular tensor categories and they are not tensor equivalent to $\Rep(u_q(\mathfrak{g}))$ for any complex simple Lie algebra $\mathfrak{g}$ in
Section 4.}
\end{remark}

\section{Main way to construct ribbon factorizable Hopf algebras}\label{sec2.2}
This section is devoted to giving a general method to construct ribbon factorizable Hopf algebras. Before this, we give some notations. Suppose $(H,R)$ is a quasitriangular Hopf algebra and $\pi:H\rightarrow K$ is a surjective Hopf map. Then it's known that $(K,(\pi\otimes \pi)(R))$ is a quasitriangular Hopf algebra. In particular, if $I$ is Hopf ideal of $H$ then $(H/I,\overline{R})$ is also a quasitriangular Hopf algebra, where $\overline{R}$ is defined through the natural quotient map $\pi(h)=h+I$ for $h\in H$. For convenience, we denote $K^+$ as $\ker\epsilon$ for a Hopf algebra $K$, here $\epsilon$ is counit of $K$.

Let $(H,R)$ be a quasitriangular Hopf algebra. For convenience, we denote the map $f_{R_{21}R}$ in \eqref{eq2.5} as $\Phi_R$ in the following content. To construct factorizable Hopf algebras, we use following lemma.
\begin{lemma}\label{lem3.0.0.x}
Let $(H, R)$ be a factorizable Hopf algebra. Assume $G\subseteq (G(H)\cap Z(H))$, where $Z(H)$ is the center of $H$. Then $(H/H(\Bbbk G)^+,\overline{R})$ is a factorizable Hopf algebra if and only if the matrix $[(\Phi_R^{-1}(g))(h)]_{g,h\in G}$ is non-degenerate, where $\overline{R}$ is induced from $R$ through the natural quotient map $\pi:H\rightarrow H/H(\Bbbk G)^+$.
\end{lemma}

\begin{proof}
Denote the Hopf algebra $H/H(\Bbbk G)^+$ as $K$. Since $K$ is quotient Hopf algebra of $H$, we can view the braided tensor category $(\Rep(K), \otimes, \Bbbk, \overline{c})$ as a braided tensor subcategory of $(\Rep(H), \otimes, \Bbbk, c)$, where $\overline{c}$ is given by $\tau \circ \overline{R}$. Let $\mathcal{D}=\Rep(H)$ and let $\mathcal{E}=\Rep(K)$. Because $(H, R)$ is a factorizable Hopf algebra, $\mathcal{D}$ is a non-degenerate finite braided tensor category. Hence we can use Theorem \ref{thm2.2} as follows. Since $G\subseteq (G(H)\cap Z(H))$, we know $\Phi_R^{-1}(g)\in G(H^*)$ for $g\in G$. Let $g\in G$ and let $\Bbbk_g$ be the one-dimensional representation of $H$ determine by $s.1=\Phi_R^{-1}(g)(s)1$ for $s\in H$. Let $\mathcal{F}$ be the full subcategory of $\mathcal{D}$ with objects
$$\Ob(\mathcal{F}) := \{\text{finite direct sums of } \{\Bbbk_g,\;g\in G\}\}.$$
By definition of $\mathcal{F}$, it is a topologizing braided tensor subcategory of $\mathcal{D}$. Let $V\in \mathcal{D}$. By directly computation, we know $V\in C_\mathcal{D}(\mathcal{F})$ if and only if $g.v=v$ for all $g\in G,\;v\in V$. i.e. $H(\Bbbk G)^+.V=0$. This is equivalent to say $\mathcal{E}=C_\mathcal{D}(\mathcal{F})$ due to the definition of $\Rep(K)$. By Theorem \ref{thm2.2}, we know that $\mathcal{F}$ is non-degenerate if and only if $\mathcal{E}$ is non-degenerate. It's not hard to see that $\mathcal{F}$ is non-degenerate if and only if the matrix $[(\Phi_R^{-1}(g))(h)]_{g,h\in G}$ is  non-degenerate, and hence we complete the proof.
\end{proof}

In order to compute more convenient, we introduce following lemma. Assume $G=\langle x_i|\;x_i^{n_i}=1, x_ix_j=x_jx_i\rangle_{1\leq i,j\leq m}$ as groups. Suppose $\eta$ is bicharacter on $G$ which is given by $\eta(x_i,x_j)=t_j^{m_{ij}}$, where $t_j$ is primitive $n_j$th root of unity. Denote the matrix $(m_{ij})_{1\leq i,j\leq m}$ as $M$.

\begin{lemma}\label{lem3.4}
The matrix $\left[\eta(g,h)\right]_{g,h\in G}$ is non-degenerated if and only if the following equation has a unique solution in $\mathbb{Z}_{n_1}\times ...\times \mathbb{Z}_{n_m}$
\begin{align*}
(i_1,...,i_m)M=(0,...,0),\;(i_1,...,i_m)\in  \mathbb{Z}_{n_1}\times ...\times \mathbb{Z}_{n_m}.
\end{align*}
\end{lemma}

\begin{proof}
Denote the dual group of $G$ as $\widehat{G}$. Define $\gamma:G\rightarrow \widehat{G}$ by $\gamma(g)(h)=\eta(g,h)$ for $g,h\in G$. Directly, we see that $\left[\eta(g,h)\right]_{g,h\in G}$ is non-degenerated if and only if $\gamma$ is injective. By definition, $\gamma$ is injective if and only if the following equation has unique solution
$$ (i_1,...,i_m)M=(0,...,0),\;(i_1,...,i_m)\in  \mathbb{Z}_{n_1}\times ...\times \mathbb{Z}_{n_m}.$$
\end{proof}
Let $m\in \mathbb{N}$. Assume $G\subseteq G(H)$ is a subgroup such that $G=\mathbb{Z}_{n_1}\times\cdots \times \mathbb{Z}_{n_m}=\langle g_1,\cdots,g_m|\;g_i^{n_i}=1,\;g_ig_j=g_jg_i,\;1\leq i,j\leq m\rangle$ as groups. Moreover, we suppose that $\overline{G}\subseteq G(H^*)$ is a subgroup such that $\overline{G}=\mathbb{Z}_{n_1}\times\cdots \times \mathbb{Z}_{n_m}=\langle \chi_1,\cdots,\chi_m|\;\chi_i^{n_i}=1,\;\chi_i\chi_j=\chi_j\chi_i,\;1\leq i,j\leq m\rangle$ as groups. Let $t_j$ be a primitive $n_j$th root of $1$ and let $\chi_i(g_j)\chi_j(g_i)=t_j^{m_{ij}}$ for $1\leq i,j\leq m$. Then we have

\begin{corollary}\label{coro3.1}
Keeping above notation, assuming $\{a_i|\;1\leq i\leq n\}\subseteq H$ and $H$ is generated by $\{a_i|\;1\leq i\leq n\}$ as an algebra. If the following conditions hold
\begin{itemize}
 \item[(i)] $\chi_i\rightharpoonup a_j\leftharpoonup \chi_i^{-1}=g_i^{-1}a_j g_i$ for $1\leq i\leq m,\;1\leq j\leq n$;
  \item[(ii)] $(i_1,...,i_m)M=(0,...,0),\;(i_1,...,i_m)\in  \mathbb{Z}_{n_1}\times ...\times \mathbb{Z}_{n_m}$ has a unique solution;
\end{itemize}
then $\langle \chi_i g_i|\;1\leq i\leq m\rangle\subseteq Z(D(H))$ and $(D(H)/I,\overline{\mathcal{R}})$ is factorizable Hopf algebra, where $I=D(H)\langle \chi_i g_i|\;1\leq i\leq m\rangle^{+}$ and $\mathcal{R}$ is the standard universal $\mathcal{R}$-matrix of $D(H)$.
\end{corollary}

\begin{proof}
By definition of $D(H)$, we know $\chi_i a_j \chi_i^{-1}=\chi_i\rightharpoonup a_j\leftharpoonup \chi_i^{-1}$. Combing this with (i), we know $\chi_i g_i\in Z(D(H))$. Let $B=\langle \chi_i g_i|\;1\leq i\leq m\rangle$ as groups. Since Lemma \ref{lem3.4} and $(\Phi_\mathcal{R}^{-1}(\chi_i g_i))(\chi_j g_j)=\chi_i(g_j)\chi_j(g_i)=t_j^{m_{ij}}$ for $1\leq i,j\leq m$, the matrix $[(\Phi_R^{-1}(g))(h)]_{g,h\in B}$ is non-degenerate. Now using Lemma \ref{lem3.0.0.x} and we get what we want.
\end{proof}

To obtain ribbon Hopf algebras, we give following propositions.

\begin{proposition}\label{lem3.r1}
Assume $(H,R)$ is a quasitriangular Hopf algebra satisfying:
\begin{itemize}
  \item[(i)] $G(H)\cap Z(H)=\{1\}$,
  \item[(ii)]  $|G(H)|=2m-1$ for some $m\in \mathbb{N}$,
  \item[(iii)] $S^{2n}=\Id$ for some odd number $n$,
\end{itemize}
then the element $\mathrm{g}^{-m}u$ is the unique ribbon element for $(H,R)$, where $\mathrm{g}=uS(u^{-1})$.
\end{proposition}

\begin{proof}
Recall that $\mathrm{g}=uS(u^{-1})$ and $S^4(h)=\mathrm{g} h \mathrm{g}^{-1}$ for $h\in H$. Define the map $\sigma:H\rightarrow H$ by $\sigma(h)=\mathrm{g}^{m}h\mathrm{g}^{-m}$. By (ii), we have $\mathrm{g}^{2m}=\mathrm{g}$ and hence $\sigma^2=(S^2)^2$. Using (ii)-(iii), we know $\sigma^{n(2m-1)}=(S^2)^{n(2m-1)}=\Id$. Note that $n(2m-1)$ is odd, we get $\sigma=S^2$. This implies that $\mathrm{g}^{m}\in G(\text{R}(H,R))$. Now applying Theorem \ref{thm2.r}, we obtain that $\mathrm{g}^{-m}u$ is a ribbon element. By (i), we know it is unique ribbon element $(H,R)$.
\end{proof}
In practice, we also need following result to obtain ribbon Hopf algebras.
\begin{proposition}\label{lem3.r2}
Assume $(H,R)$ is a quasitriangular Hopf algebra and $H$ is generated by $\{a_i|\;1\leq i\leq n\}$ as an algebra. If the following conditions hold
\begin{itemize}
  \item[(i)] $G(H)\cap Z(H)=\{1\}$,
  \item[(ii)]  $S^2(a_i)=\lambda_i a_i$ for all $1\leq i\leq n$, where $\lambda_i$ is root of unity with odd order,
\end{itemize}
then there is a unique ribbon element for $(H,R)$.
\end{proposition}

\begin{proof}
By (ii), we can find some $m\in \mathbb{N}$ such that $\lambda_i^{2m-1}=1$ for all $1\leq i\leq n$. We claim that $\mathrm{g}^{-m}\in G(\text{R}(H,R))$. Using $S^4(a_i)=\mathrm{g} a_i\mathrm{g}^{-1}$ and $S^4(a_i)=\lambda_i^2 a_i$, we have $\mathrm{g}a_i\mathrm{g}^{-1}=\lambda_i^2 a_i$. This implies that $\mathrm{g}^ma_i\mathrm{g}^{-m}=\lambda_i a_i$. Hence $S^2(a_i)=\mathrm{g}^ma_i\mathrm{g}^{-m}$. Since $H$ is generated by $\{a_i|\;1\leq i\leq n\}$ as an algebra, $S^2(h)=\mathrm{g}^m h\mathrm{g}^{-m}$ for $h\in H$. Using this formula, we obtain $S^4(h)=(\mathrm{g}^{m})^2 h(\mathrm{g}^{-m})^2$. Because $S^4(h)=\mathrm{g} h \mathrm{g}^{-1}$ and the condition (i), we get $(\mathrm{g}^{m})^2=\mathrm{g}$. So the claim holds. Using Theorem \ref{thm2.r}, we know that $\mathrm{g}^{-m}u$ is a ribbon element. By (i), it is the unique ribbon element for $(H,R)$.
\end{proof}
Now we give following theorem to construct ribbon factorizable Hopf algebras. Assume $\{a_i|\;1\leq i\leq n\}\subseteq H$ and $H$ is generated by $\{a_i|\;1\leq i\leq n\}$ as an algebra. Then we have

\begin{theorem}\label{thm3.rf}
Let $r\in \mathbb{N}^+$. If the conditions (i)-(ii) of Corollary \ref{coro3.1} hold and
\begin{itemize}
 \item[(i)] $G(H^*)=\langle \chi_i|\;1\leq i\leq m\rangle$ as group;
  \item[(ii)] $G(H)\cap Z(H)=\{1\}$,
  \item[(ii)]  $S^2(a_i)=\lambda_i a_i$ for all $1\leq i\leq n$, where the order of $\lambda_i$ is $2r-1$,
\end{itemize}
then $(D(H)/I,\overline{\mathcal{R}})$ is factorizable Hopf algebra with unique ribbon element $\mathrm{g}^{-r}u$, where $I=D(H)\langle \chi_i g_i|\;1\leq i\leq m\rangle^{+}$ and $\mathcal{R}$ is the standard universal $\mathcal{R}$-matrix of $D(H)$.
\end{theorem}

\begin{proof}
By Corollary \ref{coro3.1}, we only need to show that $\mathrm{g}^{-r}u$ is the unique ribbon element for $(D(H)/I,\overline{\mathcal{R}})$. Denote $D(H)/I$ as $K$ for simplicity. By definition of $K$, we have $G(K)=G(H)$. Note that $H\subseteq K$ as Hopf algebras, so $G(K)\cap Z(K)\subseteq G(H)\cap Z(H)$. Combining this with (ii), we get $G(K)\cap Z(K)=\{1\}$. Due to the proof of Lemma \ref{lem3.r2}, we know $(\mathrm{g}^{r})^2=\mathrm{g}$ and $(S_H)^2(h)=\mathrm{g}^r h\mathrm{g}^{-r}$ for $h\in H$. This implies that $(S_{H^*})^{-2}(f)=\mathrm{g}^{r} \rightharpoonup f\leftharpoonup \mathrm{g}^{-r}$ for $f\in H^*$. Thus $(S_{D(H)})^2(x)=\mathrm{g}^r x\mathrm{g}^{-r}$ for $x\in D(H)$. Since $K=D(H)/I$, we get that $(S_K)^2(k)=\mathrm{g}^r k\mathrm{g}^{-r}$ for $k\in K$. Hence $\mathrm{g}^{-r}\in G(R(K,\overline{\mathcal{R}}))$. Using Theorem \ref{thm2.r}, we know that $\mathrm{g}^{-r}u$ is a ribbon element. Due to (ii), the ribbon element is unique.
\end{proof}

Another way we will use to construct ribbon factorizable Hopf algebras is given as follows.
\begin{theorem}\label{thm3.4}
If the conditions (i)-(ii) of Corollary \ref{coro3.1} hold and
\begin{itemize}
 \item[(i)] $G(H^*)=\langle \chi_i|\;1\leq i\leq m\rangle$ as group;
  \item[(ii)] $G(H)\cap Z(H)=\{1\}$,
  \item[(ii)]  There is $g_0\in G(H)$ such that $S^2(x)=g_0x{g_0}^{-1}$ for all $x\in H$,
\end{itemize}
then $(D(H)/I,\overline{\mathcal{R}})$ is factorizable Hopf algebra with unique ribbon element $g_0^{-1}u$, where $I=D(H)\langle \chi_i g_i|\;1\leq i\leq m\rangle^{+}$ and $\mathcal{R}$ is the standard universal $\mathcal{R}$-matrix of $D(H)$.
\end{theorem}

\begin{proof}
Denote $D(H)/I$ as $K$. Using similar discussion with the proof of Theorem \ref{thm3.rf}, we have $G(K)\cap Z(K)=\{1\}$ and $(S_{D(H)})^2(x)=g_0x{g_0}^{-1}$ for $x\in D(H)$. Hence $(S_{K})^2(x)=g_0x{g_0}^{-1}$ for $x\in K$. Since $(K,\overline{\mathcal{R}})$ is a quasitriangular, $S^4(k)=\mathrm{g} k \mathrm{g}^{-1}$ for $k\in K$. Because $G(K)\cap Z(K)=\{1\}$ and $S^4(k)=g_0^2 k g_0^{-2}$ for $k\in K$, so $\mathrm{g}=g_0^2$. Now using Theorem \ref{thm2.r}, we know that $g_0^{-1}u$ is a ribbon element. Due to $G(K)\cap Z(K)=\{1\}$, it is unique.
\end{proof}

\section{Some applications}
\subsection{Ribbon factorizable Hopf algebras $D(u(\mathcal{D},\lambda,\mu))/I$.}
In this subsection, we will apply Theorem \ref{thm3.rf} to $u(\mathcal{D},\lambda,\mu)$ and hence obtain some ribbon factorizable Hopf algebras. Given a datum $\mathcal{D}=\mathcal{D}(G, (g_i)_{1\leq i\leq \theta}, (\chi_i)_{1\leq i\leq \theta}, (a_{ij})_{1\leq i,j\leq \theta})$ and parameters  $\lambda,\mu $ for $\mathcal{D}$. Then we have

\begin{lemma}\label{thm4.g}
Suppose $\{\overline{\chi_i}|\;1\leq i\leq \theta\}\subseteq \widehat{G}$ satisfying $\overline{\chi_i}(g_j)=\chi_j(g_i)$ for $1\leq i,j\leq \theta$. Assume $1\leq i\leq \theta$. Then there is a unique element $\overline{\chi_i}'\in G(u(\mathcal{D},\lambda,\mu)^*)$ satisfying $\overline{\chi_i}'(g)=\overline{\chi_i}(g)$ for $g\in G$.
\end{lemma}

\begin{proof}
Let $1\leq i\leq \theta$. Define $\overline{\chi_i}'\in G(u(\mathcal{D},\lambda,\mu)^*)$ by $\overline{\chi_i}'(g)=\overline{\chi_i}(g)$ and $\overline{\chi_i}'(x_j)=0$ for $g\in G,\;1\leq j\leq \theta$. Next, we show $\overline{\chi_i}'$ is well defined.

Directly, $\overline{\chi_i}'$ keeps the relations \eqref{eq2.1.x}-\eqref{eq2.2.x}. Assume $\lambda_{kl}\neq 0$. By assumption about $\lambda$, we have $\chi_k\chi_l=1$. So $\chi_k\chi_l(g_i)=1$ for $1\leq k\leq \theta$. This implies that $\overline{\chi_i}(g_kg_l)=1$. Hence $\overline{\chi_i}'$ keeps the relation \eqref{eq2.3.x}. Let $\alpha_k$ be a simple root of $\Phi$ and suppose $\mu_{\alpha_k}\neq 0$. By assumption about $\mu$, we have $\chi_k^{N_k}=1$. Since $\overline{\chi_i}(g_k)^{N_k}=\chi_k^{N_k}(g_i)=1$, $\overline{\chi_i}$ keep the relation \eqref{eq2.4.x} for $\alpha=\alpha_k$. Note that $x_\alpha$ and $u_\alpha(\mu)$ are iterated defined by $x_{\alpha_1},...x_{\alpha_\theta}$ and $\mu_{\alpha_1},...\mu_{\alpha_\theta}$ for $\alpha\in \Phi^+$, hence $\overline{\chi_i}'$ keeps the relation \eqref{eq2.4.x} for all $\alpha\in \Phi^+$. Since $\overline{\chi_i}'$ keep the relation $g_kx_kg_k^{-1}=q_{kk}x_k$, we know $\overline{\chi_i}'(x_k)=0$. Hence $\overline{\chi_i}'$ is the only element in $G(u(\mathcal{D},\lambda,\mu)^*)$ satisfying $\overline{\chi_i}'(g)=\overline{\chi_i}(g)$ for $g\in G$.
\end{proof}

\begin{remark}
\emph{Suppose $\{\overline{\chi_i}|\;1\leq i\leq \theta\}\subseteq \widehat{G}$ satisfying $\overline{\chi_i}(g_j)=\chi_j(g_i)$ for $1\leq i,j\leq \theta$. Since above lemma, we will denote the $\overline{\chi_i}'$ as $\overline{\chi_i}$ with no confusion.}
\end{remark}

Denote the Cartan matrix $(a_{ij})_{1\leq i,j\leq \theta}$ as $A$. Then we have
\begin{theorem}\label{thm4.2}
Let $H=u(\mathcal{D},\lambda,\mu)$. Suppose $gcd(\ord(\chi_ig_i),\det(A))=1$ for all $1\leq i\leq \theta$ and there is $\{\overline{\chi_i}|\;1\leq i\leq \theta\}\subseteq \widehat{G}$ satisfying $\overline{\chi_i}(g_j)=\chi_j(g_i)$ for $1\leq i,j\leq \theta$. Then
\begin{itemize}
 \item[(i)] $\{\overline{\chi_i} g_i|\;1\leq i\leq \theta\}\subseteq Z(D(H))$;
  \item[(ii)] $(D(H)/I,\overline{\mathcal{R}})$ is a factorizable Hopf algebra, where $I=D(H)\langle \overline{\chi_i} g_i|\;1\leq i\leq \theta\rangle^{+}$;
\end{itemize}
Moreover, if $\widehat{G}=\langle \chi_i|\;1\leq i\leq \theta \rangle$ and $G(H^*)=\langle \overline{\chi_i}|\;1\leq i\leq \theta \rangle$ as groups, then $(D(H)/I,\overline{\mathcal{R}})$ is ribbon factorizable Hopf algebra with unique ribbon element.
\end{theorem}

\begin{proof}
Directly, we have $\overline{\chi_i}\rightharpoonup x_j\leftharpoonup \overline{\chi_i}^{-1}=g_i^{-1}x_j g_i=q_{ij}^{-1}x_j$. Thus we have (i). Let $B=\langle\overline{\chi_i} g_i|\;1\leq i\leq \theta\rangle$ as groups. To show (ii), we only need to prove that the matrix $[(\Phi_R^{-1}(g))(h)]_{g,h\in B}$ is non-degenerate, where $\Phi_R:D(H)^*\rightarrow D(H)$ is given in Lemma \ref{lem3.0.0.x}. Let $n_i=\ord(\chi_ig_i)$ for $1\leq i\leq \theta$. Let $s=\prod_{i=1}^\theta (\overline{\chi_i} g_i)^{k_i}$, where $0\leq k_i <n_i$. Assume that $\Phi_R^{-1}(s)(h)=1$ for all $h\in B$. Then we will show that $s=1$ and hence complete the proof. Since $\Phi_R^{-1}(s)(\overline{\chi_j} g_j)=1$, we get $\sum_{i=1}^\theta a_{ji}k_i=0$. Using $A^*A=\det(A)I$ and $(\ord(\chi_ig_i),\det(A))=1$, where $A^*$ is the adjoint matrix of $A$, we obtain $n_i|k_i$ for all $1\leq i\leq \theta$. This implies that $s=1$ and so (ii) holds.

Suppose $\widehat{G}=\langle \chi_i|\;1\leq i\leq \theta \rangle$ and $G(H^*)=\langle \overline{\chi_i}|\;1\leq i\leq \theta \rangle$ as groups. To show remain parts, we only need to prove that the conditions of Theorem \ref{thm3.rf} hold for $H$. Since $G(H^*)=\langle \overline{\chi_i}|\;1\leq i\leq \theta \rangle$, the condition (i) of Theorem \ref{thm3.rf} holds for $H$. Using $\widehat{G}=\langle \chi_i|\;1\leq i\leq \theta \rangle$, we know the condition (ii) of Theorem \ref{thm3.rf} holds. By definition, we have $S_H^2(x_i)=q_{ii}^{-1}x_i$. Since the order of $q_{ii}$ is odd for all $1\leq i\leq \theta$, we can find $r\in \mathbb{N}^+$ such that the order of $q_{ii}$ is $2r-1$. Hence the condition (iii) of Theorem \ref{thm3.rf} holds.
\end{proof}

\subsection{Ribbon factorizable Hopf algebras $D(\mathcal{U}(\mathcal{D}_{red}))/I$.}
Next, we will use Theorem \ref{thm3.rf} to study Hopf algebras $\mathcal{U}(\mathcal{D}_{red})$. Let $G$ be a finite abelian group. Now fix a reduced YD-datum $\mathcal{D}_{red}=(f_i, g_i, \chi_i)_{1\leq i\leq \theta}$, where $f_i,g_i\in G,\chi_i\in \widehat{G}$, $1\leq i \leq \theta$ such that
$$q_{ij}=\chi_j(g_i)=\chi_i(f_j),\;f_ig_i\neq 1,\;1\leq i,j \leq \theta.$$
Define
\begin{align*}
&V=\oplus_{i=1}^\theta \Bbbk v_i \in _G^G\mathcal{YD}, \quad \quad \text{with basis } v_i\in V_{g_i}^{\chi_i},1\leq i \leq \theta,\\
&W=\oplus_{i=1}^\theta \Bbbk w_i \in _G^G\mathcal{YD}, \quad \;\; \text{with basis } w_i\in W_{f_i}^{\chi_i^{-1}},1\leq i \leq \theta.
\end{align*}

Let $\ord(\chi_if_i)=n_i$ for $1\leq i \leq \theta$. Since our goal, we assume that $\langle \chi_if_i|\;1\leq i\leq \theta \rangle\cong \mathbb{Z}_{n_1}\times ...\times \mathbb{Z}_{n_\theta}$ in the following content. Let $t_j$ be a primitive $n_j$th root of $1$ and let $\chi_i(g_j)\chi_j(g_i)=t_j^{m_{ij}}$ for $1\leq i,j\leq m$. Then we have
\begin{theorem}\label{thm4.5}
If the following conditions hold for $\mathcal{U}(\mathcal{D}_{red})$
\begin{itemize}
\item[(i)] $\mathfrak{B}(V)$ is finite dimensional;
\item[(ii)] $\widehat{G}=\langle \chi_i|\;1\leq i\leq \theta \rangle$ as group;
\item[(iii)] $(i_1,...,i_m)M=(0,...,0),\;(i_1,...,i_m)\in  \mathbb{Z}_{n_1}\times ...\times \mathbb{Z}_{n_m}$ has a unique solution, here the matrix $M=(m_{ij})_{1\leq i,j\leq \theta}$;
\end{itemize}
then $(\mathcal{U}(\mathcal{D}_{red}),\overline{\mathcal{R}})$ is factorizable Hopf algebra with a unique ribbon element.
\end{theorem}

\begin{proof}
Since $\mathfrak{B}(V)$ is finite dimensional, the $\overline{\mathcal{R}}$ is well defined. Due to (ii)-(iii), we know the conditions of Theorem \ref{thm3.rf} hold for $\mathcal{U}(\mathcal{D}_{red})$. Hence $(\mathcal{U}(\mathcal{D}_{red}),\overline{\mathcal{R}})$ is factorizable Hopf algebra with a unique ribbon element.
\end{proof}

As a corollary, we have
\begin{corollary}\label{coro4.3}
The Hopf algebra $K(\alpha,n)$ is a factorizable Hopf algebra with unique ribbon element.
\end{corollary}

\begin{proof}
Directly, we know the conditions of Theorem \ref{thm4.5} hold. Hence we complete the proof.
\end{proof}
For the case that the order of $q_{ii}$ is not odd for some $1\leq i\leq \theta$, we introduce another theorem as follows.  Fix a reduced YD-datum $\mathcal{D}_{red}=(f_i, g_i, \chi_i)_{1\leq i\leq \theta}$, where $f_i,g_i\in G,\chi_i\in \widehat{G}$, $1\leq i \leq \theta$ such that
$$q_{ij}=\chi_j(g_i)=\chi_i(f_j),\;f_ig_i\neq 1,\;1\leq i,j \leq \theta.$$

\begin{theorem}\label{thm4.6}
If the following conditions hold for $\mathcal{U}(\mathcal{D}_{red})$
\begin{itemize}
\item[(i)] $\mathfrak{B}(V)$ is finite dimensional;
\item[(ii)] $G=\mathbb{Z}_{n}^\theta=\langle g_1,\cdots,g_\theta|\;g_i^{n}=1,\;g_ig_j=g_jg_i,\;1\leq i,j\leq \theta\rangle$ as groups;
\item[(iii)] $\widehat{G}=\mathbb{Z}_{n}^\theta=\langle \chi_1,\cdots,\chi_\theta|\;\chi_i^{n}=1,\;\chi_i\chi_j=\chi_j\chi_i,\;1\leq i,j\leq \theta\rangle$ as groups.
\item[(iv)] Let $\omega$ be a primitive $n$th root of $1$ and let $\chi_j(g_i)=\omega^{m_{ij}}$ for $1\leq i,j\leq \theta$. And we have $(\det(M),n)=1$ and $(\det(M+M^t),n)=1$, where $M=(m_{ij})_{1\leq i,j\leq \theta}$ and $M^t$ is the transpose matrix of $M$;
\end{itemize}
then $(\mathcal{U}(\mathcal{D}_{red}),\overline{\mathcal{R}})$ is factorizable Hopf algebra with a unique ribbon element.
\end{theorem}

\begin{proof}
Since $\mathfrak{B}(V)$ is finite dimensional, the $\overline{\mathcal{R}}$ is well defined. Using Theorem \ref{thm2.5}, we obtain that $\mathcal{U}(\mathcal{D}_{red})\cong D(A)/\langle \chi_if_i-1|\;1\leq i\leq \theta\rangle$ as Hopf algebras. Using $\chi_j(g_i)=\chi_i(f_j)$ for $1\leq i\leq \theta$ and the assumption about $G, \widehat{G}$, we obtain that $G$ can be also written
$$G=\langle f_1,\cdots,f_\theta|\;f_i^{n}=1,\;f_if_j=f_jf_i,\;1\leq i,j\leq \theta\rangle.$$
Since $(\det(M+M^t),n)=1$, the following equation has unique solution
\begin{align*}
(i_1,...,i_m)(M+M^t)=(0,...,0),\;(i_1,...,i_m)\in  \mathbb{Z}_{n}^\theta.
\end{align*}
Hence we can use Lemma \ref{lem3.0.0.x} to get that $(\mathcal{U}(\mathcal{D}_{red}),\overline{\mathcal{R}})$ is factorizable Hopf algebra.

To show that it is a ribbon Hopf algebra with unique ribbon element, we will use Theorem \ref{thm3.4}. Since $G(A^*)=\widehat{G}=\langle \chi_1,\cdots,\chi_\theta\;1\leq i\leq \theta\rangle$ as groups, we know the condition (i) of Theorem \ref{thm3.4} holds. Using (iii) and the definition of $A$, we know that $G(A)\cap Z(A)={1}$. Let $q_{ii}=\omega^{t_i}$ for $1\leq i\leq \theta$. Due to (iii), the following equation has a unique solution
\begin{align*}
(i_1,...,i_m)M=(-t_1,...,-t_m),\;(i_1,...,i_m)\in  \mathbb{Z}_{n}^\theta,
\end{align*}
then we denote the solution as $(s_1,...,s_m)$. Define $g_0=\prod_{i=1}^m g_i^{s_i}$. Then $S^2(g_0)=g_0 v_i g_0^{-1}$ for all $1\leq i\leq \theta$. Hence $S^2(x)=g_0x{g_0}^{-1}$ for all $x\in A$. Now we can use Theorem \ref{thm3.4} to obtain that $\mathcal{U}(\mathcal{D}_{red})$ is a ribbon Hopf algebra with unique ribbon element.
\end{proof}

As a corollary, we have
\begin{corollary}\label{coro4.4}
If $(n,21)=1$, then the Hopf algebra $H(\omega,n)$ is a factorizable Hopf algebra with unique ribbon element.
\end{corollary}

\begin{proof}
Using $(n,21)=1$, we can check that the conditions of Theorem \ref{thm4.6} hold for $H(\omega,n)$. Hence we get what we want.
\end{proof}

\subsection{Property of prime and comparison.}
By Theorem \ref{thm2.3}, we know that the set $\{x_{\beta_1}^{i_1}...x_{\beta_p}^{i_p}g|\;0\leq i_k\leq N_k,\;1\leq k\leq p,\;g\in G\}$ is a linear basis of $u(\mathcal{D},\lambda,\mu)$. Let $1\leq i\leq \theta$. Recall that we have assumed that $\{\alpha_1,...,\alpha_\theta\}$ is a system of simple roots in subsection 2.3. Define $y_i\in u(\mathcal{D},\lambda,\mu)^*$ by
$$y_i:x_{\beta_1}^{i_1}...x_{\beta_p}^{i_p}g\mapsto \delta_{i_1\beta_1+...i_p \beta_p,\alpha_i}, \text{where } 0\leq i_k\leq N_k,\;1\leq k\leq p,\;g\in G. $$

Then we have
\begin{lemma}\label{lem4.2}
Assume $\chi_i$ can be extended to an element in $G(u(\mathcal{D},\lambda,\mu)^*)$ for all $1\leq i\leq \theta$. Then the following equalities hold in the Hopf algebra $D(u(\mathcal{D},\lambda,\mu))$
\begin{align}
\label{eq4.1} \Delta(y_i)=1\otimes y_i+y_i\otimes \chi_i,\;x_iy_j-y_jx_i=\delta_{i,j}(\chi_i-g_i),\;1\leq i,j\leq \theta.
\end{align}
\end{lemma}

\begin{proof}
By Theorem \ref{thm2.4}, $\{X_{\beta_1}^{i_1}...X_{\beta_p}^{i_p}g|\;i_k\geq 0,\;1\leq k\leq p,\;g\in G\}$ is a linear basis of $U(\mathcal{D},\lambda)$. Thus we can define $Y_i\in U(\mathcal{D},\lambda)^*$ by
$$Y_i:X_{\beta_1}^{i_1}...X_{\beta_p}^{i_p}g\mapsto \delta_{i_1\beta_1+...i_p \beta_p,\alpha_i}, \text{where } 0\leq i_k\leq N_k,\;1\leq k\leq p,\;g\in G. $$
Let $\chi_i':U(\mathcal{D},\lambda)\rightarrow \Bbbk$ be the algebra map given by $\chi_i'(g)=\chi_i(g),\;\chi_i'(X_j)=0$ for $1\leq j\leq \theta$. Since the multiplication of $U(\mathcal{D},\lambda)$ keeps the degree, then it's not hard to see that $Y_i(ab)=Y_i(a)\epsilon(b)+\chi_i'(a)Y_i(b)$. Thus $Y_i\in U(\mathcal{D},\lambda)^o$ (the finite dual of $U(\mathcal{D},\lambda)$), and $\Delta(Y_i)=Y_i\otimes \epsilon +\chi_i'\otimes Y_i$. Let $I=\langle X_\alpha^{N_J}-u_\alpha(\mu)|\;\alpha\in \Phi_J^+,J\in T\rangle$. Using the fact that $u_\alpha(\mu)$ belongs to the augmentation ideal of $\Bbbk[g_i^{N_i}|\;1\leq i\leq \theta]$, we know that $Y_i(I)=0$. Since $\chi_i$ can be extended to an element in $G(u(\mathcal{D},\lambda,\mu)^*)$, i.e. $\chi_i(u_\alpha(\mu))=0$ for all $\alpha\in \Phi_J^{+},J\in T$. This implies that $\chi_i'(I)=0$. Let $\pi:U(\mathcal{D},\lambda)\rightarrow u(\mathcal{D},\lambda,\mu)$ be the natural quotient map. Then the finite dual map $\pi^o:u(\mathcal{D},\lambda,\mu)\hookrightarrow U(\mathcal{D},\lambda)^o$ satisfies $\pi^o(y_i)=Y_i$ and $\pi^o(\chi_i)=\chi_i'$ due to above discussion. But we already shown that $\Delta(Y_i)=Y_i\otimes \epsilon +\chi_i'\otimes Y_i$, so we have $\Delta(y_i)=y_i\otimes \epsilon +\chi_i\otimes y_i$ in $u(\mathcal{D},\lambda,\mu)^*$. This implies $\Delta(y_i)=1\otimes y_i+y_i\otimes \chi_i$ in $D(u(\mathcal{D},\lambda,\mu))$.

By definition of $D(u(\mathcal{D},\lambda,\mu))$, we get
$$x_iy_j=[g_i\rightharpoonup y_j\leftharpoonup S^{-1}(x_i)]g_i+(g_i\rightharpoonup y_j) x_i+ (x_i\rightharpoonup y_j)g_i.$$
To complete the proof, we need only to check the following equations hold
$$[g_i\rightharpoonup y_j\leftharpoonup S^{-1}(x_i)]=-\delta_{i,j}1,
\;(g_i\rightharpoonup y_j)=y_j,\;(x_i\rightharpoonup y_j)=\delta_{i,j}\chi_i.$$
By definition, $S^{-1}(x_i)=-x_ig_i^{-1}$. Directly, we have
$$[g_i\rightharpoonup y_j\leftharpoonup S^{-1}(x_i)](x_{\beta_1}^{i_1}...x_{\beta_p}^{i_p}g)=-y_i(x_ix_{\beta_1}^{i_1}...x_{\beta_p}^{i_p}g)\chi_\gamma(g_i^{-1}),$$ where $\gamma=i_1\beta_1+...i_p \beta_p$. Using the definition of $y_i$, we know that $$y_i(x_ix_{\beta_1}^{i_1}...x_{\beta_p}^{i_p}g)\chi_\gamma(g_i^{-1})=\delta_{i,j}\prod_{k=1}^p\delta_{i_k,0}.$$
Hence $[g_i\rightharpoonup y_j\leftharpoonup S^{-1}(x_i)]=-\delta_{i,j}1$. Similarly, one can check the other equalities.
\end{proof}
It's known that Hopf algebras $u(\mathcal{D},0,0)$ satisfy the conditions of above lemma. Next, we give following result to compare different modular tensor categories arising from Hopf algebras.
Since our goal, we assume the datum $\mathcal{D}$ satisfies the following conditions in this subsection
\begin{itemize}
\item[(i)] $\widehat{G}=\langle \chi_i|\;1\leq i\leq \theta \rangle$ as groups;
\item[(ii)] $\widehat{G}=\langle \overline{\chi_i}|\;1\leq i\leq \theta \rangle$ as groups and $\overline{\chi_i}(g_j)=\chi_j(g_i)$ for $1\leq i,j\leq \theta$
\item[(iii)] The Dynkin diagram of $(a_{ij})_{1\leq i,j\leq \theta}$ is connected;
\end{itemize}
Let $H=u(\mathcal{D},0,0)$ and $I=D(H)\langle \overline{\chi_i} g_i|\;1\leq i\leq \theta\rangle^{+}$. By Theorem \ref{thm4.2}, we know that $(D(H)/I,\overline{\mathcal{R}})$ is ribbon factorizable Hopf algebra. For convenience, we denote $D(H)/I$ as $H(\mathcal{D})$.

\begin{theorem}\label{thm4.4}
The set of quotient Hopf algebras of $H(\mathcal{D})$ are the same as the set of quotient Hopf algebras of $\Bbbk [G(H(\mathcal{D}))/S]$, where $S=\langle \chi_i^{-1}g_i|\;1\leq i\leq \theta\rangle$ as groups.
\end{theorem}

\begin{proof}
Let $I_0=\langle(1-\chi_i^{-1}g_i),x_i,y_i|\;1\leq i\leq \theta\rangle $ as ideal. Suppose $0\neq I_1\subseteq H(\mathcal{D})$ is a Hopf ideal. Then we will show $I_0\subseteq I_1$. Since $u(\mathcal{D},0,0)^*=\langle \chi_i,y_i|\;1\leq i\leq \theta \rangle$ as algebras, we know $H(\mathcal{D})$ is pointed Hopf algebra. By assumption, $\widehat{G}=\langle \overline{\chi_i}|\;1\leq i\leq \theta \rangle$ as groups. Hence we have $G(H(\mathcal{D}))=G$. This implies that all the skew primitive elements of  $H(\mathcal{D})$ belong to the set $\{ g-h, gx_i,gy_i|\;1\leq i\leq \theta,g,h\in G\}$. Since $I_1\neq 0$ is a coideal of $H(\mathcal{D},\lambda,\mu)$, we know that $I_1$ must contain a skew primitive element, i.e. $(1-g)\in I_1 $ or $x_i\in I_1$ or $y_i\in I_1$ for some $g\in G, 1\leq i\leq \theta$. We claim that there is $x_i\in I_1$. Assume $(1-g)\in I_1 $. Since $\widehat{G}=\langle \chi_i|\;1\leq i\leq \theta \rangle$ as groups, we can find some $\chi_i\in G$ such that $\chi_i(g)\neq 1$. Note that $(1-\chi_i(g))(x_i)=(1-g)x_i+gx_i(1-g^{-1})$, hence $x_i\in I_1$. Similarly, if $y_i\in I_1$, then $(1-\chi_i^{-1}g_i)\in I_1$ due to the relation \eqref{eq4.1} in Lemma \ref{lem4.2}. Using $(\chi_i^{-1}g_i)(x_i)(\chi_i^{-1}g_i)^{-1}=q_{ii}^2 x_i$, we get that $x_i\in I_1$ and so the claim holds. Now we assume $x_i\in I_1$. Using the relation \eqref{eq4.1} in Lemma \ref{lem4.2} again, we obtain that $(1-\chi_i^{-1}g_i)\in I_1$. Note that $(\chi_i^{-1}g_i)(y_i)(\chi_i^{-1}g_i)^{-1}=q_{ii}^{-2} y_i$, hence $y_i\in I_1$. So we have shown $\{(1-\chi_i^{-1}g_i),x_i,y_i\}\subseteq I_1$. Using $(\chi_i^{-1}g_i)(x_j)(\chi_i^{-1}g_i)^{-1}=q_{ii}^{a_{ij}} x_j$ and $(\chi_i^{-1}g_i)(y_j)(\chi_i^{-1}g_i)^{-1}=q_{ii}^{-a_{ij}} y_j$, we know that $\{(1-\chi_j^{-1}g_j),x_j,y_j\}\subseteq I_1$ if $a_{ij}\neq 0$. By assumption, the Dynkin diagram of $(a_{ij})_{1\leq i,j\leq \theta}$ is connected. Hence $I_0\subseteq I_1$. This implies that $H(\mathcal{D})/I_1$ is a quotient of $H(\mathcal{D})/I_0$. Note that $H(\mathcal{D})/I_0\cong \Bbbk [G(H(\mathcal{D}))/S]$ as Hopf algebras, and so we have complete the proof.
\end{proof}
As a corollary, we get
\begin{corollary}\label{coro4.2}
For the Hopf algebra $H(\mathcal{D})$, if $\overline{\chi_i}(g_j)=\chi_j(g_i)$ for $1\leq i,j\leq \theta$ and $G=\langle g_i^2|\;1\leq i\leq \theta \rangle$ as groups, then $\Rep(H(\mathcal{D}))$ is prime modular tensor category.
\end{corollary}

\begin{proof}
Using above theorem, we know that the set of quotient Hopf algebras of $H(\mathcal{D})$ are the same as the set of quotient Hopf algebras of $\Bbbk [G(H(\mathcal{D}))/S]$, where $S=\langle \chi_i^{-1}g_i|\;1\leq i\leq \theta\rangle$ as groups. Note that $\overline{\chi_i}=\chi_i$ in this case, hence $S=\langle g_i^2|\;1\leq i\leq \theta \rangle$. By assumption about $G$, we know $S=G$. Since $G(H(\mathcal{D}))=G$, we know that quotient Hopf algebras of $H(\mathcal{D})$ are trivial. Hence $\Rep(H(\mathcal{D}))$ is prime modular tensor category.
\end{proof}

\begin{remark}
\emph{In fact, the authors \cite{LC} have characterized the conditions for $H(\mathcal{D})$ to be a ribbon factorizable Hopf algebra by using other methods. In this Corollary, we also considered the prime property.}
\end{remark}

We say that a symmetric square matrix is \emph{connected} if it can't be written into the following form by permutating its rows or permutating its lines:
$$\begin{pmatrix} A & 0 \\ 0 & B \end{pmatrix}.$$
Given a reduced YD-datum $\mathcal{D}_{red}=(f_i, g_i, \chi_i)_{1\leq i\leq \theta}$, where $f_i,g_i\in G,\chi_i\in \widehat{G}$, $1\leq i \leq \theta$. Recall that we have assumed that $\langle \chi_if_i|\;1\leq i\leq \theta \rangle\cong \mathbb{Z}_{n_1}\times ...\times \mathbb{Z}_{n_\theta}$. Let $t_j$ be a primitive $n_j$th root of $1$ and let $\chi_i(g_j)\chi_j(g_i)=t_j^{m_{ij}}$ for $1\leq i,j\leq m$. Define $M=(m_{ij})_{1\leq i,j\leq m}$. Then we have
\begin{theorem}
Assume that $\mathfrak{B}(V)$ is finite dimensional and the above matrix $M$ is connected, then all quotient Hopf algebras of $\mathcal{U}(\mathcal{D}_{red})$ are the same as the set of all quotient Hopf algebras of $\Bbbk [G/S]$, where $S=\langle f_ig_i|\;1\leq i\leq \theta\rangle$ as groups.
\end{theorem}

\begin{proof}
Using similar discussions to the proof of Theorem \ref{thm4.4}, we obtain what we want.
\end{proof}

As a corollary, we have
\begin{corollary}
The modular tensor category $\Rep (H(\alpha,m,n))$ is prime and it's not tensor equivalent to $\Rep(u_q(\mathfrak{g}))$ for any complex simple Lie algebra $\mathfrak{g}$.
\end{corollary}

\begin{proof}
By a similar way to the proof of Corollary \ref{coro4.2}, we get what we want.
\end{proof}

For the case that the order of $q_{ii}$ is not odd for some $1\leq i\leq \theta$, we introduce another theorem as follows.  Fix a reduced YD-datum $\mathcal{D}_{red}=(f_i, g_i, \chi_i)_{1\leq i\leq \theta}$, where $f_i,g_i\in G,\chi_i\in \widehat{G}$, $1\leq i \leq \theta$ such that
$$q_{ij}=\chi_j(g_i)=\chi_i(f_j),\;f_ig_i\neq 1,\;1\leq i,j \leq \theta.$$

\begin{theorem}
If the conditions of Theorem \ref{thm4.6} hold for $\mathcal{U}(\mathcal{D}_{red})$
and the matrix $\mathfrak{q}$ is connected, then $\Rep (\mathcal{U}(\mathcal{D}_{red}))$ is a modular tensor category.
\end{theorem}

\begin{proof}
Since Theorem \ref{thm4.6}, we only need to show $\Rep (\mathcal{U}(\mathcal{D}_{red}))$ is prime. Note that $\chi_j(g_i)=\chi_i(f_j)$ for $1\leq i\leq \theta$ and $\widehat{G}=\langle \chi_1,\cdots,\chi_\theta\;1\leq i\leq \theta\rangle$ as groups, we get $G=\langle g_if_i\;1\leq i\leq \theta\rangle$. Using similar discussion to the proof of Theorem \ref{thm4.4}, we know that all quotient Hopf algebras of $\mathcal{U}(\mathcal{D}_{red})$ are trivial. It implies that $\Rep (\mathcal{U}(\mathcal{D}_{red}))$ is prime.
\end{proof}

As a corollary, we have
\begin{corollary}
If $(n,21)=1$, then the modular tensor category $\Rep H(\omega,n)$ is prime and it's not tensor equivalent to $\Rep(u_q(\mathfrak{g}))$ for any complex simple Lie algebra $\mathfrak{g}$.
\end{corollary}

\begin{proof}
Since $(n,21)=1$, we know the conditions of Theorem \ref{thm4.6} hold for $H(\omega,n)$. Hence $\Rep (H(\omega,n))$ is a prime modular tensor category. Because $\dim(H(\omega,n))=(2n)^8$ and hence it's not tensor equivalent to $\Rep(u_q(\mathfrak{g}))$ for any complex simple Lie algebra $\mathfrak{g}$.
\end{proof}

Lastly, we will recall the ribbon factorizable Hopf algebras given in \cite[Example 5.18]{LC} and compare them with our examples. Let $q$ be a primitive $2n$-th root of unity, for $n\geq 1$ an odd integer. The Hopf algebra $\text{Drin}_{K^*}(\mathfrak{B}_q, \mathfrak{B}_q^*)$ is generated by $\{x_i,y_i,k_i|\;1\leq i \leq 2\}$ as an algebra, subject to relations, for $1\leq i\neq j\leq 2$,
\begin{align*}
&k_ik_j=k_jk_i,\;k_i^{2n}=1,\;
k_ix_i=x_ik_i,\; k_iy_i=y_ik_i,\;k_ix_j =qx_jk_i,\;k_iy_j=q^{-1}y_jk_i,\\
&x_iy_j + y_jx_i = \delta_{i,j}(1-k_i),\;
x_i^2=y_i^2=0,\; (x_1x_2-x_2x_1)^{2n}=(y_2y_1-y_1y_2)^{2n}=0,
\end{align*}

The coproduct, counit and antipode are given by $\Delta(k_i) = k_i\otimes k_i$ and
\begin{align*}
\Delta(x_1) = x_1\otimes 1 + k_2^n\otimes x_1, \;\Delta(x_2)=x_2\otimes 1 + k_1^nk_2 \otimes x_2,\\
\Delta(y_1) = y_1\otimes 1 + k_1k_2^n \otimes  y_1,\;\Delta(y_2) = y_2 \otimes 1 + k_1^n \otimes y_2,\\
\epsilon(k_i)=1, \;\epsilon(x_i)=\epsilon(y_i)=0,\\
S(k_i)=k_i^{-1}, S(x_1) = -k_2^{-n}x_1,S(x_2) = -k_1^{-n}k_2^{-1}x_2,\\
S(y_1) =-k_1^{-1}k_2^{-n}y_1,\;S(y_2) = -k_1^{-n}y_2.
\end{align*}

\begin{remark}
\emph{Since the dimension of $\text{Drin}_{K^*}(\mathfrak{B}_q, \mathfrak{B}_q^*)$ is $256n^4$ (\cite[Example 5.18]{LC}) and $\dim(K(\alpha,m))=3^{10} m^6$ for some odd number $m$, their representation categories are not equivalent. Because $\text{Drin}_{K^*}(\mathfrak{B}_q, \mathfrak{B}_q^*)$ admits four different ribbon structures (\cite[Example 5.18]{LC}) and $H(\omega,n)$ only admits one ribbon structure, their representation categories are also not equivalent as braided tensor categories.}
\end{remark}

\subsection{Prime modular tensor categories $\Rep(D(G_{m,n}))$ and semisimple ribbon factorizable Hopf algebras $A(p,q)$.}
Let $m, n$ be two prime numbers such that $m\equiv 1(\text{mod}\;n)$. Assume $l\in \mathbb{N}$ satisfying $l^{n}\equiv 1(\text{mod}\;m)$ and $m \not \equiv 1(\text{mod}\;m)$. Note that we don't ask $m,n$ are odd numbers here. Define $G_{m,n}=\langle a,b|\;a^{m}=b^{n}=1,bab^{-1}=a^{l}\rangle$ as a group. Then we will show that $\Rep(D(G_{m,n}))$ are prime modular tensor categories. Before this, we need the following lemma.

\begin{lemma}\label{lem4.3}
Let $r,s$ be two distinct prime numbers. Then semisimple and factorizable Hopf algebra with dimension $rs^2$ are group algebras.
\end{lemma}

\begin{proof}
Assume $H$ is semisimple quasitriangular Hopf algebra with dimension $rs^2$ satisfying it is non-trivial. Suppose $r,s$ are odd numbers. Since \cite[Corollary 9.7]{ED}, \cite[Theorem 3.10.1]{Na2} and \cite[Theorem 0.1]{Na3}, we know $H$ is not simple Hopf algebra. Hence we can use \cite[Theorem 1.1-1.2]{ZG} to see that $H$ is not factorizable Hopf algebra. Hence there is no semisimple factorizable Hopf algebra with dimension $rs^2$ such that it is non-trivial. 

Assume $r=2$. Using \cite[Theorem 3.12.4]{Na2}, we know $H$ is isomorphic to $\mathscr{B}_0$ or $\mathscr{B}_0^*$ given in \cite[Section 1]{Na2}. Suppose $H$ is factorizable Hopf algebra, then $H$ has the property of $G(H^*)|G(H)$. But both $\mathscr{B}_0$ and $\mathscr{B}_0^*$ doesn't satisfy this property. Hence $H$ must not a factorizable Hopf algebra. For the remain case $s=2$, there are only two isomorphism classes $\mathscr{A}_0$, $\mathscr{A}_1$ given in \cite{Ges}. Using similar discussion to \cite[Proposition 4.11]{ZG}, we know that they are not factorizable Hopf algebras. Hence we have completed the proof.
\end{proof}
Using above lemma, we can obtain a family of semisimple prime modular tensor categories which are not equivalent to $\Rep(G)$ as tensor categories for any group $G$.
\begin{proposition}\label{pro4.3}
$\Rep(D(G_{m,n}))$ are prime moudular tensor categories
\end{proposition}

\begin{proof}
We need only to show that there is no quotient Hopf algebra of $D(G_{m,n})$ such that it is factorizable Hopf algebra. Let $K$ be a quotient Hopf algebra of $D(G_{m,n})$ such that $1< \dim(K)< \dim(H)$. Assume $K$ is not a group algebra, we claim that $K$ is not a factorizable Hopf algebra. Since $K$ is not a group algebra and \cite{Na4}, we know $\dim(K)\notin\{r,s,rs\}$. Thus $\dim(K)\in\{rs^2,r^2s\}$. By Lemma \ref{lem4.3}, $K$ is not a factorizable Hopf algebra. Thus $K$ is an abelian group algebra if $K$ is factorizable Hopf algebra. 

Note that $Z(D(G_{m,n}))\cap G(D(G_{m,n}))=\{\alpha\otimes 1|\;\alpha:G_{m,n}\rightarrow \Bbbk \text{ is a group homomorphism}\}$. Hence $G(D(G_{m,n})^*)=\{1\otimes \alpha|\;\alpha:G_{m,n}\rightarrow \Bbbk \text{ is a group homomorphism}\}$. Assume $K$ is factorizable Hopf algebra. Then $K$ is abelian group algebra due to above discussion. Let $\langle,\rangle$ be the braiding structure corresponding with the standard universal $\mathcal{R}$-matrix of $D(G_{m,n})$. Then $K$ is factorizable Hopf algebra if and only if $\langle,\rangle|_{K^*\times K^*}$ is non-degenerated. Since $G(K^*)\subseteq G(D(G_{m,n})^*)$ and $G(D(G_{m,n})^*)=\{1\otimes \alpha|\;\alpha:G_{m,n}\rightarrow \Bbbk \text{ is a group homomorphism}\}$, we know $\langle,\rangle|_{K^*\times K^*}$ is always degenerated. But this contradicts with previous result. Thus $K$ is not factorizable and we completed the proof.
\end{proof}

Next, we will apply Corollary \ref{coro3.1} to construct a family of semisimple factorizable Hopf algebras $A(p,q)$. Although $\Rep(A(p,q))$ is not prime, we will prove that $A(p,q)$ can't be obtained by the obvious way, i.e. it's not tensor product of trivial Hopf algebras (group algebras or their dual) and Drinfeld doubles. Let $\chi:\mathscr{A}_l\rightarrow \Bbbk$ be the algebra map which is determined by $\chi(e_g)=\delta_{g,b}$ and $\chi(x)=\omega$, where $\omega$ is primitive $q$-th root of unity. Then we have
\begin{proposition}\label{pro3.1.x}
The Hopf algebra $(D(\mathscr{A}_0)/I,\overline{\mathcal{R}})$ is a ribbon factorizable Hopf algebra with dimension $p^2q^3$, where $I=D(\mathscr{A}_0)\langle \chi x\rangle^{+}$.
\end{proposition}

\begin{proof}
Directly, we have $\chi \rightharpoonup e_g\leftharpoonup \chi^{-1}=x^{-1}e_g x$ and the matrix $(\chi(x)^{2ij})_{1\leq i,j\leq q}$ is non-degenerated. Thus the conditions of Corollary \ref{coro3.1} hold. Hence we get what we want.
\end{proof}

\begin{remark}\label{rk2}
\emph{Let $p, q$ be two odd prime numbers such that $p\equiv 1(\text{mod}\;q)$. For convenience, the Hopf algebra $D(\mathscr{A}_0)/I$ above will be denoted by $A(p,q)$ and we will discuss it in the next subsection.}
\end{remark}

Then we will describe $A(p,q)$ by using generators and relations. Moreover, we give their universal $\mathcal{R}$-matrices which make them to be ribbon factorizable Hopf algebras explicitly.

\begin{proposition}\label{thm2.x}
The semisimple Hopf algebra $A(p,q)$ is generated by $\{x,y,z_i,e_g|\;1\leq i \leq q,\;g\in G\}$ as algebra, with the relations
$$
yx=xy^t,\;z_ix=xz_i,\;e_gx=xe_{g\triangleleft x},\;z_iy=yz_i,$$

$$
 e_gy=y\sum_{1\leq i\leq q} z_i (e_{a^{-1}\triangleleft x^i g a}),\;e_gz_i=z_ie_g,\;x^q=y^p=1,\;z_iz_j=\delta_{i,j}z_i,\;e_ge_h=\delta_{g,h}e_g.$$
The quasitriangular structure on $A(p,q)$ given by Proposition \ref{pro3.1.x} is
\begin{equation}\label{eqr}
\overline{\mathcal{R}}=\sum_{1\leq i \leq p,1\leq j,k \leq q}\omega^{jk} e_{a^ib^j}x^k\otimes y^i z_k x^{j}.
\end{equation}

The coproduct, counit, involution and antipode are given by
$$
\Delta(x) = x\otimes x, \;\Delta(y) =\sum_{1\leq i\leq q} y^{t^i}\otimes yz_i,
$$

$$
\Delta(z_i) =\sum_{1\leq j\leq q} z_j \otimes z_{i-j},\;\Delta(e_g) =\sum_{h\in G} e_h \otimes e_{h^{-1}g},
$$

$$
\epsilon(x)=\epsilon(y)=1, \;\epsilon(z_i)=\delta_{i,0},\;\epsilon(e_g)=\delta_{g,1},
$$

$$
S(x) =x^{-1},\; S(y) =\sum_{1\leq i\leq q}y^{-t^{-i}}z_i, \;S(z_i) =z_{-i},\;S(e_g) =e_{g^{-1}}.
$$
\end{proposition}

\begin{remark}
\emph{From this proposition, we know that the Hopf subalgebra $B=\langle z_i|\;1\leq i\leq l \rangle$ belongs to center of $A(p,q)$. Hence $(A(p,q)/I,\overline{R})$ give a quotient Hopf algebra, where $I=A(p,q)B^{+}$. Moreover, we can see that $(A(p,q)/I,\overline{R})$ is factorizable Hopf algebra. Now applying decomposition theorem in \cite{LC}, we know that $\Rep A(p,q)=\Rep(B)\boxtimes \Rep(A(p,q)/I)$ as modular tensor categories. Since dimension reason, $\Rep(B)$ is prime. To see $\Rep(A(p,q)/I)$ is also a prime modular tensor category, note that $(A(p,q)/I,\overline{R})$ is actually isomorphic to $(D(G_{p,q}),\mathcal{R})$ by direct computation, and hence $\Rep(A(p,q)/I)$ is prime by Proposition \ref{pro4.3}. Hence we have given a decomposition for $\Rep A(p,q)$.}
\end{remark}
To show above proposition, we introduce some lemmas as follows. Assume $\{E_{g;x^i}|\;g\in G, 1\leq i\leq q\}$ is the dual basis of $\{e_{g}x^i|\;g\in G, 1\leq i\leq q\}$ for $\mathscr{A}_0^*$, i.e. $E_{g;x^i}(e_h x^j)=\delta_{g,h}\delta_{i,j}$ for $g,h\in G$ and $1\leq i,j\leq q$. Let $\omega$ be a primitive $q$-th root of unity. To find algebraic generator of $\mathscr{A}_0^*$, we define $Y,\;Z_i(1\leq i\leq q),\;\chi$ as follows: $$Y=\sum_{1\leq j\leq q}E_{a;x^j},\;Z_i=E_{1,x^i},\;\chi=\sum_{1\leq j\leq q}\omega^jE_{b;x^j}.$$

\begin{lemma}\label{lem3.2.x}
The following statements hold for $\mathscr{A}_0^*$
\begin{itemize}
  \item[(i)] $(E_{g;x^i})(E_{h;x^j})=\delta_{i,j}E_{gh;x^i}$,
  \item[(ii)]  $\Delta(E_{g;x^i})=\sum_{1\leq j\leq q} E_{g;x^j}\otimes E_{g\triangleleft x^j;x^{i-j}}$,
  \item[(iii)] $E_{a^ib^j;x^k}=\omega^{-jk}Y^i \chi^j Z_k$.
\end{itemize}
\end{lemma}

\begin{proof}
By definition of $\mathscr{A}_0$, we have $\Delta(e_kx^l)=\sum_{u\in G}e_ux^l\otimes e_{u^{-1}k}x^l$. This implies $(E_{g;x^i}E_{h;x^j})(e_kx^l)=\delta_{i,j}\delta_{gh,k}\delta_{i,l}$. Hence we have (i).

Since $(e_hx^j)(e_kx^l)=(\delta_{h,k\triangleleft x^{-j}}e_h x^{j+l})$, we get
$E_{g;x^i}[(e_hx^j)(e_kx^l)]=\delta_{h,k\triangleleft x^{-j}}\delta_{g,h}\delta_{i,j+l}$. This implies (ii).

Due to (i), we know $Y^i \chi^j Z_k=\omega^{jk}E_{a^ib^j;x^k}$. And so (iii) holds.
\end{proof}
As a corollary of Lemma \ref{lem3.2.x}, we have
\begin{corollary}\label{lem3.3.x}
The Hopf algebra $\mathscr{A}_0^*$ is generated by $\{Y,Z_i,\chi|\;1\leq i\leq q\}$ as algebra, with relations
$$
Z_iY=YZ_i,\;\chi Y=Y^t \chi,\;\chi Z_i=Z_i\chi,\;Y^p=\chi^q=1,\;Z_iZ_j=\delta_{i,j}Z_i,$$
The coproduct, counit, involution and antipode are given by
$$
\Delta(Y) = \sum_{1\leq i\leq q} YZ_i\otimes Y^{t^i}, \;\Delta(Z_i) =\sum_{1\leq j\leq q} Z_j\otimes Z_{i-j},\;\Delta(\chi) =\chi\otimes \chi,
$$

$$
\epsilon(Y)=\epsilon(\chi)=1, \;\epsilon(Z_i)=\delta_{i,0},
$$

$$
S(Y) =\sum_{1\leq i\leq q}Y^{-t^{-i}}Z_i, \;S(Z_i) =Z_{-i},\;S(\chi) =\chi^{-1}.
$$
\end{corollary}

\begin{proof}
By Lemma \ref{lem3.2.x}, the only non-trivial things are the following equalities:
$$\chi Y=Y^t \chi,\;\Delta(Y) = \sum_{1\leq i\leq q} YZ_i\otimes Y^{t^i},\;S(Y) =\sum_{1\leq i\leq q}Y^{-t^{-i}}Z_i.$$

Since Lemma \ref{lem3.2.x}, we have $\chi Y=\sum_{1\leq i\leq q}\omega^i E_{ba;x^i}$ and $Y^t\chi=\sum_{1\leq i\leq q}\omega^i E_{a^tb;x^i}$. By definition of $G$, $ba=a^tb$ and hence $\chi Y=Y^t \chi$.

Using Lemma \ref{lem3.2.x} again, we have $\Delta(Y) = \sum_{1\leq i,j\leq q} E_{a;x^j}\otimes E_{a^{t^j};x^{i-j}}$. Note that $E_{a;x^j}=YZ_j$ and $E_{a^{t^j};x^{i-j}}=Y^{t^j}Z_{i-j}$ by Lemma \ref{lem3.2.x}, thus $\Delta(Y) = \sum_{1\leq i,j\leq q} YZ_j\otimes Y^{t^j}Z_{i-j}$. Since $\sum_{1\leq j\leq q}Z_{i-j}=1$, we know $\Delta(Y) = \sum_{1\leq i\leq q} YZ_i\otimes Y^{t^i}$.

By definition, $E_{a;x^i}(S(e_g x^j))=E_{a;x^i}(e_{g\triangleleft x^j} x^{-j})=\delta_{g,a\triangleleft x^i}\delta_{j,-i}$. Hence $S(E_{a;x^i})=E_{a\triangleleft x^i;x^{-i}}$. Note that $E_{a\triangleleft x^i;x^{-i}}=Y^{t^i}Z_{-i}$, thus $S(Y) =\sum_{1\leq i\leq q}Y^{-t^{-i}}Z_i$.
\end{proof}
Now we can show proposition \ref{thm2.x}. For simplicity, we write $a$ as $a+I$ in $A(p,q)$.

\textbf{Proof of Proposition \ref{thm2.x}.}
Denote $y=Y$ and $z_i=Z_i$. By definition of $A(p,q)$, we have $\chi=x^{-1}$. Combine this with Corollary \ref{lem3.3.x}, we know $A(p,q)$ is generated by $\{x,y,z_i,e_g|\;1\leq i \leq q,\;g\in G\}$ as algebra. Let $\mathcal{R}$ be the standard universal $\mathcal{R}$-matrix of $D(\mathscr{A}_0)$. Because $\{e_{a^ib^j}x^k|;1\leq i \leq p,1\leq j,k \leq q\}$ and $\{E_{a^ib^j;x^k}|;1\leq i \leq p,1\leq j,k \leq q\}$ are dual basis, we get $\mathcal{R}=\sum_{1\leq i \leq p,1\leq j,k \leq q}e_{a^ib^j}x^k\otimes E_{a^ib^j;x^k}$. By Lemma \ref{lem3.2.x}, we have $E_{a^ib^j;x^k}=\omega^{-jk}Y^i \chi^j Z_k$ in $D(\mathscr{A}_0)$. Since $y=Y,\;z_i=Z_i$ and $\chi=x^{-1}$ in $A(p,q)$, we know the $\overline{\mathcal{R}}$ is given by the equation \eqref{eqr}.

To complete the proof, due to Corollary \ref{lem3.3.x} and the definition of $\mathscr{A}_0$, we only need to prove the following equalities :
$$yx=xy^t,\;z_ix=xz_i,\;e_gy=y\sum_{1\leq i\leq q} z_i (e_{a^{-1}\triangleleft x^i g a}),\;e_gz_i=z_ie_g.$$

By Corollary \ref{lem3.3.x}, we have $\chi Y=Y^t \chi$ and $Z_i \chi= \chi Z_i$. Since $\chi=x^{-1}$ in $A(p,q)$, we get $yx=xy^t,\;z_ix=xz_i$. By definition of $A(p,q)$, we get
$$e_g Y=\sum_{1\leq i\leq q}\sum_{hkl=g}(e_h\rightharpoonup E_{a;x^i}\leftharpoonup e_{l^{-1}})e_k.$$
Since $(e_h\rightharpoonup E_{a;x^i}\leftharpoonup e_{l^{-1}})(e_s x^r)=\delta_{s,a}\delta_{r,i}\delta_{l,a^{-1}}\delta_{h,a\triangleleft x^i}$, we get $(e_h\rightharpoonup E_{a;x^i}\leftharpoonup e_{l^{-1}})=\delta_{l,a^{-1}}\delta_{h,a\triangleleft x^i}E_{a;x^i}$. Note that $E_{a;x^i}=YZ_i$, hence $e_gy=y\sum_{1\leq i\leq q} z_i (e_{a^{-1}\triangleleft x^i g a})$. Similarly, we have $e_gZ_i=Z_ie_g$. And so $e_gz_i=z_ie_g$.
\qed

We call a Hopf algebra \emph{trivial} if it is a group algebra or its dual. Then we will prove the following theorem.
\begin{theorem}\label{pro4.2.x}
The ribbon factorizable Hopf algebra $A(p,q)$ is not tensor product of trivial Hopf algebras and Drinfel'd doubles.
\end{theorem}
To do this, we introduce following lemma. Denote the center of $H$ as $Z(H)$.

\begin{lemma}\label{lem4.1}
Assume $G(H)$ is an abelian group and $(G(H^\ast)\cap Z(H^\ast))|_{G(H)\cap Z(H)}$ is trivial, i.e. $\alpha(g)=1$ for $\alpha\in G(H^\ast)\cap Z(H^\ast),\;g\in G(H)\cap Z(H)$. Then $H$ is not the form $K\otimes \Bbbk[G]$, where $K$ is a Hopf algebra and $G$ is a finite group with order$\geq $2.
\end{lemma}

\begin{proof}
Suppose $H=K\otimes \Bbbk[G]$ for some Hopf algebra $K$ and some finite group $G$ with order$\geq $2. Since $G(H)$ is an abelian group, we know $G$ is also an abelian group. Thus one can check that $G\subseteq Z(H)$. Since the order of $G$ is greater than one, we can find a non-trivial character $\alpha$ on $G$. Obviously $\epsilon\otimes \alpha\in G(H^\ast)\cap Z(H^\ast)$ and $(\epsilon\otimes \alpha)|_{1\otimes G}\neq \{1\}$. This contradicts with our assumption and hence we complete the proof.
\end{proof}

As a corollary, we have:
\begin{corollary} \label{coro4.1}
The Hopf algebras $A(p,q)$ are not the form $K\otimes \Bbbk[G]$, where $K$ is a Hopf algebra and $G$ is a finite group with order$\geq $2.
\end{corollary}

\begin{proof}
It's not hard to see $G(A(p,q))$ are abelian groups. Next, we will show $(G(H^\ast)\cap Z(H^\ast))|_{G(H)\cap Z(H)}=\{
1\}$ for $H=A(p,q)$ and so we complete the proof by Lemma \ref{lem4.1}.

To do above thing, we will compute $G(H)\cap Z(H)$ firstly. Recall $A(p,q)=D(\mathscr{A}_l)/I$, where $I=D(\mathscr{A}_l)\langle \chi x\rangle^{+}$. Hence one can get $|G(H)|=q^3$ through its definition. For simplicity, we also write $y+I$ as $y$ in $A(p,q)$. Denote the dual basis $\{e_gx^i\}_{g\in G, 1\leq i\leq q}$ of $\mathscr{A}_l$ as $\{E_{e_g;x^i}\}_{g\in G, 1\leq i\leq q}$. To determine $G(H)$, we define $g,h,k\in H$ as follows:
$$g:=\sum_{1\leq i\leq q}\omega^i E_{1;x^i},\; h:=\sum_{1\leq i\leq q}\omega^i E_{b;x^i},\;k=\sum_{1\leq i,j\leq q}\omega^j e_{a^ib^j}.$$
We claim $G(H)=\langle g,h,k\rangle$ as group. For simplicity, we keep the notation of Proposition \ref{thm2.x}. Since the proof of Proposition \ref{thm2.x}, we know that $g=\sum_{1\leq i\leq q}\omega^i z_i,\;h=x^{-1}$. Hence $g,h\in G(H)$. Moreover, it's not hard to see $k\in G(H)$ and  $\{g,h,k\}$ are commute with each other. By definition of $A(p,q)$, we know $\{g^ih^j k^l\}_{1\leq i,j,l\leq q}$ is linear independent. Hence the group $\langle g,h,k\rangle$ has order $q^3$. Since $|G(H)|=q^3$, the claim holds. Next we show $G(H)\cap Z(H)=\langle g,k\rangle$ as group. Note that $E_{a;x}h^i\neq h^iE_{a;x}$ for $1\leq i\leq q$, hence $G(H)\cap Z(H)\subseteq \langle g,k\rangle$. It's not difficult to know $g,k\in Z(H)$, so $G(H)\cap Z(H)=\langle g,k\rangle$ as group. Now let $\beta\in G(H^\ast)\cap Z(H^\ast)$. Define the map $i':\mathscr{A}_l^\ast \rightarrow A(p,q)$ as the following composition $\mathscr{A}_l^\ast\xrightarrow[]{\iota'} D(\mathscr{A}_l)  \xrightarrow[]{\pi'} A(p,q)$, where $\iota'$ is the natural inclusion and $\pi'$ is the natural quotient map. Using the definition of $A(p,q)$, we know that $i'$ is an injective map. If we view the linear function $\beta \circ i'$ as an element of $\mathscr{A}_l$, then we have $\beta \circ i'\in G(\mathscr{A}_l) \cap Z(\mathscr{A}_l)$ since $i'$ is an injective map. It's easy to see $G(\mathscr{A}_l) \cap Z(\mathscr{A}_l)$ is generated by $\sum_{1\leq i,j\leq q}\omega^j e_{a^ib^j}$. This implies $\beta(E_{1;x^i})=\delta_{i,0}$. Hence $\beta(g)=1$. Since $\{e_s\}_{s\in G}$ is an orthonormal family of idempotents, we can assume $\beta(e_s)=\delta_{s,g_0}$ for some $g_0\in G$. Since $e_sx=xe_{s\triangleleft x}$ and $x^p=1$, we know $g_0=b^i$ for some $1\leq i\leq q$. Using $E_{a;1}^n=E_{1;1}$ and $e_{g_0} E_{a;1}=E_{a;1}e_{a^{-1}g_0 a}$ in $A(p,q)$, we know $\beta(e_s)=\delta_{s,1}$. Hence $\beta(k)=1$ by definition of $k$. So $(G(H^\ast)\cap Z(H^\ast))|_{G(H)\cap Z(H)}=\{
1\}$.
\end{proof}

Now we can prove the main result of this subsection.\\
\textbf{Proof of Theorem \ref{pro4.2.x}.} Since dimension reason and Corollary \ref{coro4.1}, we get what we want.
\qed



\end{document}